\documentclass[a4paper,reqno,10pt]{amsart}
\usepackage{amsfonts}
\usepackage{amsmath}
\usepackage{amssymb, enumitem}
\usepackage{bbm}
\usepackage{amscd}
\usepackage[all]{xy}
\usepackage[left=2cm,right=2cm,bottom=3cm,top=3cm]{geometry}
\usepackage[hidelinks]{hyperref}
\usepackage{amsmath}
\usepackage{amssymb}
\usepackage{amsthm}

\newcounter{theoremintro}
\newtheorem{thmintro}[theoremintro]{Theorem}
\newtheorem{cnjintro}[theoremintro]{Conjecture}
\newtheorem{theorem}{Theorem}[section]
\newtheorem*{theorem*}{Theorem}

\newtheorem*{claim*}{Claim}

\newtheorem{corollary}[theorem]{Corollary}

\newtheorem{example}[theorem]{Example}

\newtheorem{notation}[theorem]{Notation}

\newtheorem{proposition}[theorem]{Proposition}

\theoremstyle{definition}

\newtheorem{remark}[theorem]{Remark}
\newtheorem{lemma}[theorem]{Lemma}
\newtheorem{definition}[theorem]{Definition}
\AtBeginDocument{   \def\MR#1{}
}

\begin{document}
\def\cprime{$'$}

\title[Actions of rigid groups on UHF-algebras]{Actions of rigid groups on
UHF-algebras}
\author[Eusebio Gardella]{Eusebio Gardella}
\address{Eusebio Gardella\\
Westf\"{a}lische Wilhelms-Universit\"{a}t M\"{u}nster, Fachbereich
Mathematik, Einsteinstrasse 62, 48149 M\"{u}nster, Germany}
\email{gardella@uni-muenster.de}
\urladdr{https://wwwmath.uni-muenster.de/u/gardella/}
\author{Martino Lupini}
\address{Mathematics Department\\
California Institute of Technology\\
1200 E. California Blvd\\
MC 253-37\\
Pasadena, CA 91125}
\date{\today }
\email{lupini@caltech.edu}
\urladdr{http://www.lupini.org/}
\subjclass[2000]{Primary 46L55, 54H05; Secondary 03E15, 37A55 }
\thanks{The first-named author was partially funded by SFB 878 \emph{Groups,
Geometry and Actions}, and by a postdoctoral fellowship from the Humboldt
Foundation. The second-named author was partially supported by the NSF Grant
DMS-1600186. This work was initiated during a visit of the authors at the
Mathematisches Forschungsinstitut Oberwolfach in August 2016 supported by an
Oberwolfach Leibnitz Fellowship. The authors gratefully acknowledge the
hospitality of the Institute. }
\keywords{Property (T), 1-cohomology, profinite group, Rokhlin property,
model action, Borel complexity, conjugacy, cocycle conjugacy, complete
analytic set}
\dedicatory{}

\begin{abstract}
Let $\Lambda$ be a countably infinite property (T) group, and let $D$ be
UHF-algebra of infinite type. We prove that there exists a continuum of
pairwise non (weakly) cocycle conjugate, strongly outer actions of $\Lambda$
on $D$. The proof consists in assigning, to any second countable abelian pro-%
$p$ group $G$, a strongly outer action of $\Lambda$ on $D$ whose (weak)
cocycle conjugacy class completely remembers the group $G$. The group $G$ is
reconstructed from the action through its (weak) 1-cohomology set endowed
with a canonical pairing function.

Our construction also shows the following stronger statement: the relations
of conjugacy, cocycle conjugacy, and weak cocycle conjugacy of strongly
outer actions of $\Lambda$ on $D$ are complete analytic sets, and in
particular not Borel. The same conclusions hold more generally when $\Lambda$
is only assumed to contain an infinite subgroup with relative property (T),
and for actions on (not necessarily simple) separable, nuclear,
UHF-absorbing, self-absorbing C*-algebras with at least one trace.

Finally, we use the techniques of this paper to construct outer actions on $%
R $ with prescribed cohomology. Precisely, for every infinite property (T)
group $\Lambda$, and for every countable abelian group $\Gamma$, we
construct an outer action of $\Lambda$ on $R$ whose 1-cohomology is
isomorphic to $\Gamma$.
\end{abstract}

\maketitle

\renewcommand*{\thetheoremintro}{\Alph{theoremintro}}

\section{Introduction}

Classification of group actions is a fundamental problem in operator
algebras, and positive results are both scarce and useful. The subject is
far more developed on the von Neumann algebra side, and it was started with
Connes' classification of periodic automorphisms on the hyperfinite II$_{1}$
factor $R$; see \cite{connes_periodic_1977}. Further generalizations to
arbitrary automorphisms \cite{connes_outer_1975} and finite group actions 
\cite{jones_actions_1980} quickly followed, and these advances culminated in
Ocneanu's work on amenable group actions on $R$ \cite{ocneanu_actions_1985}.
A consequence of his results is that for any amenable group $\Lambda$, any
two outer actions of $\Lambda$ on $R$ are cocycle conjugate. A converse to
Ocneanu's theorem was proved by Jones in \cite{jones_converse_1983}, and
this result was considerably strengthened in a recent work by Brothier and
Vaes in \cite{brothier_families_2015}, building on \cite{popa_some_2006}. We
summarize these results in the following rather strong dichotomy for outer
actions on $R$:

\begin{theorem*}
(Connes, Jones, Ocneanu, Brothier-Vaes). Let $\Lambda$ be a countable group.

\begin{enumerate}
\item If $\Lambda$ is amenable, then any two outer actions of $\Lambda$ on $%
R $ are cocycle conjugate.

\item If $\Lambda$ is not amenable, then there exist uncountably many
non-cocycle conjugate outer actions of $\Lambda$ on $R$. In fact, the
relation of cocycle conjugacy of such actions is complete analytic.
\end{enumerate}
\end{theorem*}

Ocneanu's work served as a motivation for exploring analogs of the
uniqueness statement in (1) in the context of C*-algebras. The first issue
is to find the appropriate C*-analog of $R$. UHF-algebras of infinite type
have historically played this role, as they can be regarded as ``strong"
C*-analogs of $R$. A ``weak" analog is the Jiang-Su algebra $\mathcal{Z}$
(see \cite{jiang_simple_1999}), which has also been studied in relation to
uniqueness of actions of certain amenable groups \cite%
{kishimoto_rohlin_1995,matui_actions_2011,szabo_strongly_2017}. This work
focuses mostly on UHF-algebras. Even though the existence of plenty of
projections makes their study easier, classification results for actions are
relatively difficult to obtain because of K-theoretical restrictions; see 
\cite{izumi_finite_2004-1}.

In \cite{bratteli_rohlin_1995}, Bratteli, Evans and Kishimoto studied a
family of outer actions of $\mathbb{Z}$ on the CAR algebra. It follows from
their results that no analog of Ocneanu's result can hold for outer actions.
However, they provided evidence for the fact that a uniqueness result may
hold if one assumes that not only the action is outer, but also its
extension to the weak closure in the GNS representation is outer (this is
called \emph{strong outerness}).

Recall that a unital C*-algebra $D$ is said to be \emph{strongly
self-absorbing} if it is infinite dimensional and there is an isomorphism $%
\varphi\colon D\to D\otimes_{\mathrm{min}}D$ which is approximately
unitarily equivalent to the first tensor factor embedding. The only known
examples of stably finite strongly self-absorbing C*-algebras are the
UHF-algebras of infinite type, and the Jiang-Su algebra $\mathcal{Z}$, and
in fact it is conjectured that the list is complete.

Several results in the literature, which are reviewed below, suggest that
the following may be true (part~(1) below has also been independently
conjectured by Szabo in \cite{szabo_strongly_2017}):

\begin{cnjintro}
\label{cnjintro} Let $D$ be a stably finite strongly self-absorbing
C*-algebra and let $\Lambda$ be a torsion-free countable group.

\begin{enumerate}
\item If $\Lambda$ is amenable, then any two strongly outer actions of $%
\Lambda$ on $D$ are cocycle conjugate.

\item If $\Lambda$ is not amenable, then there exist uncountably many
non-cocycle conjugate strongly outer actions of $\Lambda$ on $D$. Even more,
the relation of cocycle conjugacy of such actions is complete analytic.
\end{enumerate}
\end{cnjintro}

The reason for excluding groups with torsion is the fact that automorphisms
of finite order, even when they are strongly outer, generate unexpected
phenomena at the level of $K$-theory which obstruct any uniqueness-type
result as in~(1). For instance, it is easy to construct $\mathbb{Z}_2$%
-actions on the CAR algebra $\otimes_{n\in\mathbb{N}}M_2$, which are
strongly outer but not cocycle conjugate. As an example, one can take the
nontrivial group element to act as the following infinite tensor products: 
\begin{equation*}
\bigotimes_{n\in\mathbb{N}}\mathrm{Ad}\left( 
\begin{array}{cc}
1 &  \\ 
& -1 \\ 
& 
\end{array}
\right) \ \ \mbox{ and } \ \ \bigotimes_{n\in\mathbb{N}}\mathrm{Ad}\left( 
\begin{array}{cccc}
1 &  &  &  \\ 
& 1 &  &  \\ 
&  & 1 &  \\ 
&  &  & -1 \\ 
&  &  & 
\end{array}
\right).
\end{equation*}

As mentioned before, the cases of $D$ being a UHF-algebra of infinite type
or the Jiang-Su algebra are the most relevant ones. Part~(1) of the
conjecture above has been confirmed in a number of particular cases: for
UHF-algebras, the case $\Lambda=\mathbb{Z}$ was proved by Kishimoto in \cite%
{kishimoto_rohlin_1995}, while the case $\Lambda=\mathbb{Z}^N$ was obtained
by Matui in \cite{matui_actions_2011}. For the Jiang-Su algebra $\mathcal{Z}$%
, the case of $\Lambda=\mathbb{Z}$ was considered by Sato in \cite%
{sato_rohlin_2010}, while Matui-Sato proved the case $\Lambda=\mathbb{Z}^{2}$
and $\Lambda=\mathbb{Z}\rtimes_{-1}\mathbb{Z}$ in \cite{matui_stability_2012}
and \cite{matui_stability_2014}. More recently, and inspired by the work of
Winter on $\mathcal{Z}$- and UHF-stable classification of C*-algebras \cite%
{winter_localizing_2014}, and for \emph{elementary} amenable groups, Szabo 
\cite{szabo_strongly_2017} reduced the case $D=\mathcal{Z}$ to the case when 
$D$ is an infinite type UHF algebra. He also showed that part (1) of
Conjecture~\ref{cnjintro} holds for a group $\Lambda $ if and only if holds
for all the finitely-generated subgroups of $\Lambda $. In particular, it
follows from this and Matui's result that part (1) of Conjecture \ref%
{cnjintro} holds when $D$ is either a UHF-algebra or $\mathcal{Z}$, and when 
$\Lambda$ is a torsion-free \emph{abelian} group.

We now turn to part~(2) in the above conjecture. It should be mentioned that
it easy to see using Jones' argument from \cite{jones_converse_1983} that,
for any nonamenable group $\Lambda$, there exist at least two strongly outer
actions of $\Lambda$ on any finite strongly self-absorbing C*-algebra.
Beyond this, nothing was known until now concerning the number of cocycle
conjugacy classes (or the complexity of the cocycle conjugacy relation) for
strongly outer actions of nonamenable groups on finite strongly
self-absorbing C*-algebras.


In the present paper, we initiate the study of actions of nonamenable groups
on UHF-algebras, and we make the first contributions to part~(2) in the
above conjecture. Our main result is as follows:

\begin{thmintro}
\label{thmintro1} (See Corollary~\ref{cor:UHFuncountable} and Corollary~\ref%
{cor:UHFanalytic}). Let $D$ be a UHF-algebra of infinite type, and let $%
\Lambda$ be a countable group containing an infinite subgroup with relative
property (T). Then there exist uncountably many non-cocycle conjugate
strongly outer actions of $\Lambda$ on $D$. Indeed, the relation of cocycle
conjugacy of such actions is complete analytic.
\end{thmintro}

(Our result holds for a more general class of not necessarily simple
C*-algebras; see Theorem~\ref{Theorem:uncountably} for the precise
statement.)

It is worth mentioning that our results cannot be derived from those of
Brothier-Vaes. First, there is no general method for producing an action on
a UHF-algebra from an action on $R$. Moreover, no obvious modification of
the construction in \cite{brothier_families_2015} seems to produce an action
on a UHF-algebra. (They use the fact that the crossed product of $R$ by a
Bernoulli shift of an amenable torsion-free group is isomorphic to $R$, and
the UHF-analog of this fact is far from true.) Even more, the actions we
construct in Theorem~\ref{thmintro1} are shown to remain cocycle
inequivalent in the weak closure of $D$. Hence, our results imply the result
of Brothier-Vaes for groups with relative property (T).



The assertion that the relation of cocycle conjugacy of free actions of $%
\Lambda $ on $A$ is a complete analytic set can be interpreted as follows.
There does not exist an explicit uniform procedure that, given two strongly
outer actions of $\Lambda$ on $D$, runs for countably many (but possibly
transfinitely many) steps, at each step testing membership in some given
open sets, and at the end decides whether the given actions are cocycle
conjugate or not. In fact, the problem of deciding whether two such actions
are cocycle conjugate is as hard as testing membership in any analytic set.
Similar conclusions hold for conjugacy and weak cocycle conjugacy. For a
more detailed discussion on this interpretation, see \cite[Section 2.4]%
{epstein_borel_2011}.

The proof of our main theorem consists in assigning, to any second countable
abelian pro-$p$ group $G$, a strongly outer action of $\Lambda $ on $D$
whose weak cocycle conjugacy class completely \textquotedblleft
remembers\textquotedblright\ the group $G$. Using Popa's superrigidity
results from \cite{popa_some_2006}, the group $G$ is reconstructed from this
action via its (weak, localized) $1$-cohomology set, endowed with a
canonical ($2$-sorted) group structure. The starting point of our
construction is a canonical model action of $G$ on the UHF-algebra $%
M_{p^{\infty }}$, which we construct in Section~\ref{Section:model}. The
rest of the construction can be seen as a C*-algebra analogue of the
construction of factors of measure-preserving Bernoulli actions due to Popa 
\cite{popa_computations_2006} and T\"{o}rnquist \cite%
{tornquist_localized_2011}; see also \cite{epstein_borel_2011}.

The methods used in this construction are not specific to our context, and
can be used to compute (weak) 1-cohomology sets in other interesting cases.
As an instance of this, the last section of this paper is devoted to
constructing actions of infinite property (T) groups on $R$ with prescribed
(weak) 1-cohomology. In this context, these cohomology sets do not have a
canonical group structure. The actions we construct are self-absorbing (in a
strong sense), and there is a canonical `pairing' function $m^{\alpha}\colon
H^1_w(\alpha)\times H^1_w(\alpha)\to H^1_w(\alpha\otimes\alpha)$; see %
\autoref{Definition:mTheta}. Even this by itself does not guarantee the
existence of a group structure, but this turns out to be the case for the
actions we construct.

More specifically, for infinite groups with property (T), we prove the
following analog of the main result of \cite{popa_computations_2006} for
actions on $R$ (the result we prove is somewhat more general):

\begin{thmintro}
Let $\Lambda$ be an infinite countable property (T) group, and let $\Gamma$
be any countable abelian group. Then there exist an outer action $%
\alpha\colon \Lambda\to \mathrm{Aut} (R)$ and bijections $\eta\colon
H^1_w(\alpha)\to \Gamma$ and $\eta^{(2)}\colon H^1_w(\alpha\otimes\alpha)\to
\Gamma$ making the following diagram commute: 
\begin{align*}
\xymatrix{ H^1_w(\alpha)\times H^1_w(\alpha)
\ar[d]_-{m^{\alpha}}\ar[rr]^-{\eta\times \eta} && \Gamma\times \Gamma
\ar[d]^-{\mathrm{multiplication}}\\
H^1_w(\alpha\otimes\alpha)\ar[rr]_-{{\eta^{(2)}}} && \Gamma }
\end{align*}
\end{thmintro}

This gives a different proof of Theorem~B of~\cite{brothier_families_2015}
in the case that $\Lambda$ has (a subgroup with the relative) property (T).
For comparison, observe that Ocneanu's result implies that all outer actions
of amenable groups on $R$ have canonically isomorphic cohomology.


In the following, all topological groups are supposed to be \emph{Hausdorff}
and \emph{second countable}. All tensor products of C*-algebras are supposed
to be minimal (also called spatial); see \cite[Section II.9]%
{blackadar_operator_2006}. If $A$ is a C*-algebra and $S$ is a finite set,
then we let $A^{\otimes S}$ be the (minimal) tensor product of a family of
copies of $A$ indexed by $S$. Similarly, when $A$ is unital and $X$ is a
countable set, then we let $A^{\otimes X}$ denote the limit of the direct
system $\left( A^{\otimes S}\right) $, where $S$ varies in the collection of
finite subsets of $X$ ordered by containment, and the connective maps are
the canonical unital *-homomorphisms $\iota _{S,T}\colon A^{\otimes
S}\rightarrow A^{\otimes T}$ for $S\subset T\subset X$. In the von
Neumann-algebraic setting, we will only consider tensor products of tracial
von Neumann algebras with respect to distinguished normal tracial states,
which we denote by $\overline{\otimes }$; see \cite[Section III.3.1]%
{blackadar_operator_2006}.

\subsubsection*{Acknowledgments}

We are grateful to Samuel Coskey, \L ukasz Grabowski, Daniel Hoff, Alexander
Kechris, Andr\'{e} Nies, Stefaan Vaes, and Stuart White for many helpful
conversations. Particularly, we would like to thank Alexander Kechris for
suggesting a proof of Proposition \ref{Proposition:free-G}, and Stuart White
for suggesting the formulation of Lemma \ref{lemma:WeakClosureR} below.
Finally, we thank the referee for their careful reading of the manuscript,
and for suggesting numerous improvements.

\section{Preliminary notions on group actions\label{Section:preliminary}}

\subsection{Actions of groups on tracial von Neumann algebras\label{Sbs:vN}}

We recall some terminology about group actions on von Neumann algebras. A 
\emph{tracial von Neumann algebra }is a pair $(M,\tau )$, where $M$ is a von
Neumann algebra and $\tau $ is a normal tracial state on $\tau $. We denote
by $\mathrm{Aut}(M,\tau )$ the group of $\tau $-preserving automorphisms of $%
M$. Let $\Lambda $ be a discrete group. An \emph{action} of $\Lambda $ on $%
(M,\tau )$ is a group homomorphism $\alpha \colon \Lambda \rightarrow 
\mathrm{Aut}(M,\tau )$. An automorphism $\theta \in \mathrm{Aut}(M,\tau )$
is said to be $\emph{inner}$ if there exists a unitary $u\in M$ with $\theta
(x)=uxu^{\ast }$ for all $x\in M$. It is said to be $\emph{outer}$ if it is
not inner, and \emph{properly outer} if for every $\theta $-invariant
projection $p\in M$, the restriction of $\theta $ to $pMp$ is outer; see 
\cite[Definition XVII.1.1]{takesaki_theory_2003}.

\begin{remark}
\label{Remark:center}As it is remarked in \cite[Section 4]%
{kerr_turbulence_2010}, in the definition of properly outer autorphism one
can equivalently only consider $\theta $-invariant \emph{central }%
projections; see also the comment after Theorem XVII.1.2 in \cite%
{takesaki_theory_2003}. In particular, an automorphism of a \emph{factor} is
properly outer if and only if it is outer.
\end{remark}

Let $\theta _{0}\in \mathrm{Aut}(M_{0},\tau _{0})$ and $\theta _{1}\in 
\mathrm{Aut}(M_{1},\tau _{1})$ be automorphisms of tracial von Neumann
algebras. It is shown in \cite[Corollary 1.12]{kallman_generalization_1969}
that, if either $\theta _{0}$ or $\theta _{1}$ is properly outer, then $%
\theta _{0}\otimes \theta _{1}$ is a properly outer automorphism of $(M_{0}%
\overline{\otimes }M_{1},\tau _{0}\otimes \tau _{1})$.

\begin{definition}
\label{Definition:wm} Let $(M,\tau)$ be a tracial von Neumann algebra and
let $\Lambda$ be a discrete group. An action $\alpha \colon \Lambda \to%
\mathrm{Aut} ( M, \tau ) $ is called:

\begin{enumerate}
\item \emph{ergodic}, if the fixed point algebra $M^{\alpha}=\left\{ x\in
M\colon \alpha _{\gamma }(x)=x \mbox{ for all } \gamma\in\Lambda\right\} $
contains only the scalar multiples of the identity; see \cite[Definition 7.3]%
{takesaki_theory_2002};

\item \emph{weakly mixing}, if for any finite subset $F\subseteq M$ and $%
\varepsilon >0$, there exists $\gamma \in \Lambda $ such that 
\begin{equation*}
\left\vert \tau (x\alpha _{\gamma }(y))-\tau (x)\tau (y)\right\vert
<\varepsilon
\end{equation*}%
for every $x,y\in F$; see \cite[Definition D.1]{vaes_rigidity_2007};

\item \emph{mixing}, if for every $a,b\in M$ one has $\tau (a\alpha _{\gamma
}(b))\rightarrow \tau (a)\tau (b)$ for $\gamma \rightarrow \infty $; see 
\cite[Definition D.1]{vaes_rigidity_2007};

\item \emph{outer}, if $\alpha _{\gamma }$ is not inner for every $\gamma
\in \Lambda \setminus \left\{ 1\right\} $;

\item \emph{free}, if $\alpha _{\gamma }$ is properly outer for every $%
\gamma \in \Lambda \setminus \{1\}$; see \cite[Subsection 4.1]%
{kerr_turbulence_2010}.
\end{enumerate}
\end{definition}

Observe that any free action is, in particular, outer. When $M$ is a factor,
the converse holds in view of Remark \ref{Remark:center}.

\begin{remark}
\label{rmk:CharMixing} An action $\alpha $ is weakly mixing if and only if
the only finite-dimensional vector subspace of $L^{2}(M, \tau )$ which is
invariant under the representation associated with $\alpha $ is the space of
scalar multiples of the identity; see \cite[Proposition 2.4.2.]%
{popa_some_2006} and \cite[Proposition D.2]{vaes_rigidity_2007}.
\end{remark}

Let $\alpha $ and $\beta $ be actions of $\Lambda $ on tracial von Neumann
algebras $(M_{0}, \tau _{0})$ and $(M_{1}, \tau _{1})$, respectively. We let 
$(M_{0}\overline{\otimes }M_{1}, \tau _{0}\otimes \tau _{1})$ be the tensor
product of $M_{0}$ and $M_{1}$ with respect to the normal tracial states $%
\tau _{0}, \tau _{1}$ \cite[Section III.3.1]{blackadar_operator_2006}.
Define $\alpha \otimes \beta \colon \Lambda \rightarrow \mathrm{Aut}(M_{0}%
\overline{\otimes }M_{1}, \tau _{0}\otimes \tau _{1})$ to be the action
given by $(\alpha\otimes \beta )_{\gamma }=\alpha _{\gamma }\otimes \beta
_{\gamma }$ for $\gamma \in \Lambda $. It is easy to check that $%
\alpha\otimes\beta$ is (weakly) mixing if both $\alpha$ and $\beta$ are.

\begin{definition}
\label{Definition:invariant}Let $\pi $ a unitary representation of $\Lambda $
on a Hilbert space $H$. Following \cite{kechris_amenable_2008}, we say that $%
\pi $ has \emph{almost invariant vectors}, and write $1_{\Lambda }\prec \pi $%
, if for every $\varepsilon >0$ and finite subset $F\subseteq \Lambda $,
there exists a vector $\xi \in H$ such that $\Vert \pi (\gamma )\xi -\xi
\Vert \leq \varepsilon $ for every $\gamma \in F$. A unitary representation $%
\pi \colon \Lambda \rightarrow U(H)$ is said to be a $c_{0}$-\emph{%
representation }if for every $\xi ,\eta \in H$, the function $\gamma \mapsto
\langle \pi (\gamma )\xi ,\eta \rangle $ belongs to $c_{0}(\Lambda )$.
\end{definition}

Let $X$ be a countable set endowed with an action of $\Lambda $. We say that
the action is \emph{amenable }if it satisfies the following \emph{F\o lner
condition}: for any finite subset $Q\subseteq \Lambda $ and $\varepsilon >0$%
, there exists a finite subset $F\subseteq X$ such that $|\gamma F\triangle
F|\leq \varepsilon |F|$ for every $\gamma \in Q$. For an action $\Lambda
\curvearrowright X$, we consider the corresponding \emph{left regular
representation }$\lambda _{X}\colon \Lambda \rightarrow U(\ell ^{2}(X))$
determined by $\lambda _{X}(\gamma )(\delta _{x})=\delta _{\gamma ^{-1}x}$
for $\gamma \in \Lambda $ and $x\in X$. Theorem 1.1 in \cite%
{kechris_amenable_2008} asserts that $\Lambda \curvearrowright X$ is
amenable if and only if $1_{\Lambda }\prec \lambda _{X}$.

For a tracial von Neumann algebra $(M, \tau )$, we denote by $(M, \tau )^{%
\overline{\otimes }X}$ the tensor product of copies of $M$ indexed by $X$
with respect to the normal tracial state $\tau $; see \cite[III.3.1]%
{blackadar_operator_2006}. Then $(M, \tau )^{\overline{\otimes }X}$ carries
a canonical trace obtained from $\tau $, which we still denote by $\tau $.
We denote by $M^{\odot X}$ the algebraic tensor product, which is dense in $%
(M, \tau )^{\overline{\otimes }X}$. If $Y$ is a subset of $X$, then we
canonically identify $M^{\overline{\otimes }Y}$ with a subalgebra of $(M,
\tau )^{\overline{\otimes }X}$, and $M^{\odot Y}$ with a subalgebra of $%
M^{\odot X}$.

\begin{notation}
\label{nota:BernoulliShiftvNA} Let $X$ be a countable set endowed with an
action $\Lambda \curvearrowright X$, and let $(M, \tau )$ be a tracial von
Neumann algebra. We denote by $\beta _{\Lambda \curvearrowright X,M}\colon
\Lambda \rightarrow \mathrm{Aut}((M, \tau )^{\overline{\otimes }X})$ the
associated Bernoulli $\left( \Lambda \curvearrowright X\right) $-action with
base $\left( M, \tau \right) $, defined by permuting the indices according
to the action of $\Lambda $ on $X$.
\end{notation}

\begin{example}
In the context above, when $M=L^{\infty }(Z, \mu )$ for a probability space $%
(Z, \mu )$ and $\tau (f)=\int fd\mu $, one has $(M, \tau )^{\overline{%
\otimes }X}=L^{\infty }(Z^{X}, \mu ^{X})$ with trace $\tau (f)=\int fd\mu
^{X}$. The action on $(M, \tau )^{\overline{\otimes }X}$ corresponds in this
case to the Bernoulli action of $\Lambda $ on $(Z^{X}, \mu ^{X})$ as
considered in \cite{kechris_amenable_2008}.
\end{example}

We denote by $\kappa $ the corresponding Koopman representation of $\Lambda $
on $L^{2}(M^{\overline{\otimes }X}, \tau )$, and by $\kappa _{0}$ the
restriction of $\kappa $ to the orthogonal complement in $L^{2}(M^{\overline{%
\otimes }X}, \tau )$ of the space of scalar multiples of the identity.

The following characterization of mixing Bernoulli actions is well know; see 
\cite[Lemma 2.4.3]{popa_some_2006} and \cite[Proposition 2.1 and Proposition
2.3]{kechris_amenable_2008} for the commutative case.

\begin{proposition}
\label{Proposition:bernoulli-wm} Let $X$ be a countable set endowed with an
action of $\Lambda $, and let $(M,\tau )$ be a tracial von Neumann algebra
with a projection $p\in M$ such that $0<\tau (p)<1$. Then the Bernoulli
action $\beta _{\Lambda \curvearrowright X,M}\colon \Lambda \rightarrow 
\mathrm{Aut}((M,\tau )^{\overline{\otimes }X})$ is mixing if and only if the
stabilizers of the action $\Lambda \curvearrowright X$ are finite.
\end{proposition}

\subsection{Actions of groups on\ C*-algebras\label{Sbs:C*}}

Let $\Lambda $ be a discrete group, and let $A$ be a unital C*-algebra.
Write $\mathrm{Aut}(A)$ for the automorphism group of $A$. An \emph{action}
of $\Lambda $ on $A$ is a group homomorphism $\alpha \colon \Lambda
\rightarrow \mathrm{Aut}(A)$. In this case, we also say that the pair $%
(A,\alpha )$ is a $\Lambda $-C*-algebra. We denote by $A^{\alpha }$ the 
\emph{fixed point algebra }$A^{\alpha }=\left\{ a\in A\colon \alpha _{\gamma
}(a)=a\mbox{ for all }\gamma \in \Lambda \right\} $. We say that elements $%
x,y$ of $A$ are \emph{equivalent modulo scalars}, and write $x=y{}\mathrm{mod%
}\mathbb{C}$, if $x=\lambda y$ for some $\lambda \in \mathbb{C}%
\setminus\{0\} $. We denote by $U(A)$ the unitary group of $A$.

\begin{definition}
Let $\alpha \colon \Lambda\to \mathrm{Aut}(A)$ be an action of a discrete
group $\Lambda$ on a unital C*-algebra $A$, and let $u\colon \Lambda\to U(A)$
be a function.

\begin{enumerate}
\item We say that $u$ is a $1$-cocycle for $\alpha $ if $u_{\gamma }\alpha
_{\gamma }(u_{\rho })=u_{\gamma \rho }$ for every $\gamma , \rho \in \Lambda 
$.

\item We say that $u$ is a \emph{weak }$1$-cocycle\emph{\ }if $u_{\gamma
}\alpha _{\gamma }(u_{\rho })=u_{\gamma \rho }\ \mathrm{mod}\mathbb{C}$ for
every $\gamma , \rho \in \Lambda $.
\end{enumerate}
\end{definition}

The notion of weak $1$-cocycles allows one to define the weak $1$-cohomology
of actions. We will mostly use it for actions on tracial von Neumann
algebras, but the definition can be given in general.

\begin{definition}
\label{df:1-cohomology} Let $\alpha \colon \Lambda \rightarrow \mathrm{Aut}%
(A)$ be an action of a discrete group $\Lambda $ on a unital C*-algebra $A$.
Following \cite{popa_some_2006}, we say that two weak $1$-cocycles $u$ and $%
u^{\prime }$ for $\alpha $ are \emph{weakly cohomologous }(or cohomologous
modulo scalars), if there exists a unitary $v\in U(A)$ such that $u_{\gamma
}^{\prime }=v^{\ast }u_{\gamma }\alpha _{\gamma }(v)\ \mathrm{mod}\mathbb{C}$
for every $\gamma \in \Lambda $. We say that $u$ is a \emph{weak coboundary }%
if it is weakly cohomologous to the weak $1$-cocycle constantly equal to $1$.

We denote by $Z_{w}^{1}(\alpha )$ the set of weak $1$-cocycles for $\alpha $%
. The relation of being weakly $1$-cohomologous is an equivalence relation
on $Z_{w}^{1}(\alpha )$, and we let $H_{w}^{1}(\alpha )$ be the
corresponding quotient set, called the \emph{weak cohomology set}. The class
of the weak $1$-cocycle $u$ will be denoted by $[u]$.
\end{definition}

Let us use the notation as in the definition above. If $A$ is abelian, then
the product of two weak $1$-cocycles for $\alpha $ is again a weak $1$%
-cocycle for $\alpha $, and thus $H_{w}^{1}(\alpha )$ can be given a
canonical group structure. In general, however, one can not define a group
operation on $H_{w}^{1}\left( \alpha \right) $ in a similar fashion. To make
up for the lack of multiplication in the $1$-cohomology set $H_{w}^{1}\left(
\alpha \right) $, we consider a natural \textquotedblleft two-sorted group
structure\textquotedblright\ on $H_{w}^{1}\left( \alpha \right) $, given by
a pairing function $H_{w}^{1}(\alpha )\times H_{w}^{1}(\alpha )\rightarrow
H_{w}^{1}(\alpha \otimes \alpha )$. Such a pairing function will be used to
encode the group operation of a given countable group.

\begin{definition}
\label{Definition:mTheta} Let $(M,\tau)$ be a tracial von Neumann algebra,
let $\alpha\colon \Lambda\to\mathrm{Aut}(M,\tau)$ be an action, and denote
by $\alpha \otimes \alpha $ be the diagonal action of $\Lambda $ on $M%
\overline{\otimes} M$. Then there is a canonical function 
\begin{equation*}
\widetilde{m}^{\alpha }\colon Z_{w}^{1}(\alpha )\times Z_{w}^{1}(\alpha
)\rightarrow Z_{w}^{1}(\alpha \otimes \alpha )
\end{equation*}%
given by $\widetilde{m}^{\alpha }(u,w)=u\otimes w$ for $u,w\in
Z_{w}^{1}(\alpha )$. Observe that if $u$ is weakly cohomologous to $w$ and $%
u^{\prime }$ is weakly cohomologous to $w^{\prime }$, then $u\otimes w$ is
weakly cohomologous to $u^{\prime }\otimes w^{\prime }$. Therefore, the map $%
\widetilde{m}^{\alpha }$ induces pairing function $m^{\alpha }\colon
H_{w}^{1}(\alpha )\times H_{w}^{1}(\alpha )\rightarrow H_{w}^{1}(\alpha
\otimes \alpha )$.
\end{definition}

Following \cite{tornquist_localized_2011} one can also define the notion of
weak cohomology set \emph{localized} to a subgroup $\Delta $ of $\Lambda $,
as follows. Say that two weak $1$-cocycles $u$ and $u^{\prime }$ for an
action $\alpha\colon \Lambda \rightarrow \mathrm{Aut}( A) $ are $\Delta $-%
\emph{locally weakly cohomologous} if there exists a unitary $v\in U( A) $
such that $u_{\gamma }^{\prime }=v^{\ast }u_{\gamma }\alpha _{\gamma }(v)\ 
\mathrm{mod}\mathbb{C}$ for every $\gamma \in \Delta $. Similarly, $u$ is a $%
\Delta $-\emph{local weak coboundary }if there exists $v\in U( A) $ such
that $u_{\gamma }=v^{\ast }\alpha _{\gamma }( v) \ \mathrm{mod}\mathbb{C}$
for every $\gamma \in \Delta $.

\begin{definition}
\label{Definition:mTheta-relative}The $\Delta $-localized weak cohomology
set $H_{\Delta ,w}^{1}( \alpha ) $ is the quotient of $Z_{w}^{1}( \alpha ) $
by the relation of being $\Delta $-locally weakly cohomologous, endowed with
the pairing function $m_{\Delta }^{\alpha }\colon H_{\Delta ,w}^{1}( \alpha
) \times H_{\Delta ,w}^{1}( \alpha ) \rightarrow H_{\Delta ,w}^{1}( \alpha
\otimes \alpha ) $, given by $m^{\alpha}_{\Delta}( [ u] ,[ u^{\prime }] )= [
u\otimes u^{\prime }]$ for $u ,v \in Z^1_w(\alpha)$.
\end{definition}

Given a weak $1$-cocycle $u$ for an action $\alpha \colon \Lambda
\rightarrow \mathrm{Aut}(A)$, one can define the \emph{cocycle perturbation }%
$\alpha ^{u}\colon \Lambda \rightarrow \mathrm{Aut}(A)$ of $\alpha $ by
setting $\alpha _{\gamma }^{u}=\mathrm{Ad}(u_{\gamma })\circ \alpha _{\gamma
}$ for every $\gamma \in \Lambda $. (The weak cocycle condition implies that 
$\alpha ^{u}$ is also an action.)

\begin{definition}
\label{Definition:conjugacy}Let $\Lambda \ $be a countable discrete group,
and let $\alpha $ and $\beta $ be actions of $\Lambda$ on unital C*-algebras 
$A$ and $B$, respectively.

\begin{enumerate}
\item We say that $\alpha $ and $\beta $ are \emph{conjugate }if there
exists an isomorphism $\psi \colon A\rightarrow B$ such that $\psi \circ
\alpha _{\gamma }=\beta _{\gamma }\circ \psi $ for every $\gamma \in \Lambda 
$,

\item We say that $\alpha $ and $\beta $ are \emph{cocycle conjugate }if $%
\beta $ is conjugate to $\alpha ^{u}$ for some $1$-cocycle $u$ for $\alpha $,

\item We say that $\alpha $ and $\beta $ are \emph{weakly cocycle conjugate }%
if $\beta $ is conjugate to $\alpha ^{u}$ for some weak $1$-cocycle $u$ for $%
\alpha $.
\end{enumerate}
\end{definition}

\begin{remark}
Let $\alpha,\beta\colon \Lambda\to \mathrm{Aut}(A)$ be actions of a discrete
group $\Lambda$ on a C*-algebra $A$. It is easy to check that if $\alpha$
and $\beta$ are (weakly) cocycle conjugate, then there is a canonical
bijection between the $\Delta$-localized (weak) 1-cohomology sets of $\alpha$
and $\beta$, for any subgroup $\Delta$ of $\Lambda$.
\end{remark}

When the actions $\alpha$ and $\beta $ are conjugate, we also say that the $%
\Lambda $-C*-algebras $( A, \alpha) $ and $(B,\beta) $ are \emph{%
equivariantly isomorphic}. An \emph{equivariant unital embedding }from $%
(A,\alpha) $ to $(B,\beta) $ is an injective unital *-homomorphism $\phi
\colon A\rightarrow B$ satisfying $\phi\circ\alpha_{\gamma}=
\beta_{\gamma}\circ\phi$ for every $\gamma \in \Lambda $.

Suppose that $A$ is a unital C*-algebra. A linear functional $\tau $ on $A$
is said to be a \emph{trace} if $\tau (1)=\Vert \tau \Vert =1$ and $\tau
(ab)=\tau (ba)$ for every $a,b\in A$. We let $\mathrm{T}(A)$ be the simplex
of traces on $A$. Suppose that $\tau $ is a trace on $A$, $\theta $ is an
automorphism of $A$, and $\alpha $ is an action of $\Lambda $ on $A$. We say
that $\tau $ is $\theta $-invariant if $\tau \circ \theta =\tau $, and that
it is $\alpha $-invariant if it is $\alpha _{\gamma }$-invariant for every $%
\gamma \in \Lambda $. If $\tau $ is $\alpha $-invariant, then we also say
that $\alpha $ is $\tau $-preserving. We let $\mathrm{T}(A)^{\alpha
}\subseteq \mathrm{T}(A)$ be the closed convex subset of $\alpha $-invariant
traces. Observe that, if $\Lambda $ is amenable, then $\mathrm{T}\left(
A\right) ^{\alpha }$ is nonempty whenever $\mathrm{T}\left( A\right) $ is
nonempty.

For a trace $\tau $ on $A$, consider the corresponding \emph{left regular
representation }$\pi _{\tau }\colon A\rightarrow B(L^{2}(A,\tau ))$ obtained
via the GNS construction.\ We let $\overline{A}^{\tau }$ be the closure of $%
\pi _{\tau }(A)$ inside $B(L^{2}(A,\tau ))$ with respect to the weak
operator topology. We regard $\overline{A}^{\tau }$ as a tracial von Neumann
algebra, endowed with the unique extension of $\tau $ to $\overline{A}^{\tau
}$. The unit ball of $A$ is dense in the unit ball of $\overline{A}^{\tau }$
with respect to the $2$-norm $\Vert a\Vert _{\tau }=\tau (a^{\ast }a)^{1/2}$
defined by $\tau $. If $\alpha $ is a $\tau $-preserving action of $\Lambda $
on $A$, then it induces a canonical action $\overline{\alpha }^{\tau }\colon
\Lambda \rightarrow \mathrm{Aut}(\overline{A}^{\tau },\tau )$.

\begin{notation}
\label{Notation:Bernoulli}As in the case of actions on tracial von Neumann
algebras, given a unital C*-algebra $A$, we denote by $\beta _{\Lambda
\curvearrowright X,A}:\Lambda \rightarrow \mathrm{Aut}(A^{\otimes X})$ the
Bernoulli $\left( \Lambda \curvearrowright X\right) $-action with base $A$
induced by an action $\Lambda \curvearrowright X$ of a countable discrete
group $\Lambda $ on a countable set $X$.
\end{notation}

The following lemma is well known, and we will use it without further
reference.

\begin{lemma}
\label{lma:TensProdWkClos} Let $A$ be a C*-algebra, let $\Lambda $ be a
countable discrete group, let $\Lambda \curvearrowright X$ be an action of $%
\Lambda $ on a countable set $X$, and let $\tau _{0}$ be a trace on $A$.
Denote by $\tau $ the trace $\tau _{0}^{\otimes X}$ on $A^{\otimes X}$. Then
the action $\overline{\beta }_{\Lambda \curvearrowright X,A}^{\tau }$ is
conjugate to the von Neumann-algebraic Bernoulli action $\beta _{\Lambda
\curvearrowright X,\overline{A}^{\tau _{0}}}$.
\end{lemma}

\begin{definition}
\label{Definition:tau-conjugacy} Let $\Lambda $ be a discrete group, let $A$
be a unital C*-algebra, and let $\tau \in \mathrm{T}(A)$.

\begin{enumerate}
\item We say that $\alpha $ is \emph{strongly outer} if for every $\gamma
\in \Lambda \setminus \{1\}$ and for every $\alpha _{\gamma }$-invariant
trace $\sigma $ on $A$, the weak extension $\overline{\alpha }_{\gamma
}^{\sigma }$ is outer \cite[Definition 2.7]{matui_stability_2012};

\item If $\alpha $ is $\tau $-preserving, then we say that $\alpha $ is (%
\emph{weakly})\emph{\ }$\tau $\emph{-mixing} if $\overline{\alpha }^{\tau
}\colon \Lambda \rightarrow \mathrm{Aut}(\overline{A}^{\tau },\tau )$ is
(weakly) mixing in the sense of Definition \ref{Definition:wm}.

\item We say that two $\tau $-preserving actions $\alpha ,\beta \colon
\Lambda \rightarrow \mathrm{Aut}\left( A\right) $ are $\tau $-\emph{conjugate%
} (respectively, \emph{cocycle }$\tau $-\emph{conjugate}, or \emph{weakly
cocycle }$\tau $-\emph{conjugate}), if $\overline{\alpha }^{\tau }$ and $%
\overline{\beta }^{\tau }$ are conjugate (respectively, \emph{cocycle} \emph{%
conjugate}, or \emph{weakly cocycle} \emph{conjugate}), in the sense of
Definition \ref{Definition:conjugacy}.
\end{enumerate}
\end{definition}

\begin{remark}
Suppose that $\alpha $ is an automorphism of a C*-algebra $A$, $\sigma $ is
an $\alpha $-invariant trace, $\overline{\alpha }$ is the canonical
extension of $\alpha $ to $\overline{A}^{\sigma }$, and $p\in \overline{A}%
^{\sigma }$ is a $\overline{\alpha }$-invariant central projection. Then
defining $\tau (x) =\sigma (px)/\sigma(p)$ gives an $\alpha $-invariant 
(normalized) trace on $A$,
such that the canonical extension of $\alpha $ to $\overline{A}^{\tau }$ can
be identified with the restriction of $\overline{\alpha }$ to $p\overline{A}%
^{\sigma }$. In view of this and Remark \ref{Remark:center}, one can
equivalently replace \textquotedblleft outer\textquotedblright\ with
\textquotedblleft properly outer\textquotedblright\ in the definition of
strongly outer action. We will tacitly use this fact in the rest of the
paper.
\end{remark}

The notion of strongly action from Definition \ref{Definition:tau-conjugacy}
recovers the notion of free action on a locally compact Hausdorff space when
one considers actions on commutative C*-algebras, as the next proposition
shows.

\begin{proposition}
Let $\Lambda \curvearrowright X$ be a topological action of a discrete group 
$\Lambda $ on a locally compact Hausdorff space $X$, and denote by $\alpha
\colon \Lambda \rightarrow \mathrm{Aut}(C_{0}(X))$ the induced action. Then $%
\alpha $ is strongly outer if and only if $\Lambda \curvearrowright X$ is
free.
\end{proposition}

\begin{proof}
Suppose that $\alpha $ is strongly outer, and let $\gamma \in \Lambda
\setminus \{1\}$. To reach a contradiction, assume that there exists $x\in X$
such that $\gamma \cdot x=x$. Then the Dirac probability measure
concentrated on $\{x\}$ is Borel and $\gamma $-invariant. This measure
induces, via integration, an $\alpha _{\gamma }$-invariant trace $\tau _{x}$
on $C_{0}(X)$. Since $\overline{C_{0}(X)}^{\tau _{x}}$ is isomorphic to $%
\mathbb{C}$, the weak extension of $\alpha _{\gamma }$ cannot be outer. This
contradiction implies that $\Lambda \curvearrowright X$ is free.

Conversely, assume that $\Lambda \curvearrowright X$ is free and let $\gamma
\in \Lambda \setminus \{1\}$. Let $\tau $ be an $\alpha _{\gamma }$%
-invariant trace on $C_{0}(X)$. Then $\tau $ is given by integration with
respect to a Borel probability measure $\mu $ on $X$ which satisfies $\mu
(\gamma \cdot U)=\mu (U)$ for every open subset $U\subseteq X$. Moreover, $%
\overline{C_{0}(X)}^{\tau }$ is isomorphic to $L^{\infty }(X,\mu )$.
Suppose, to reach a contradiction, that $\overline{\alpha }_{\gamma }$ is
inner (and hence trivial). It follows that the set of fixed point of $\alpha
_{\gamma }$ has $\mu $-measure $1$ and, in particular, it is nonempty. This
is a contradiction.
\end{proof}

Let $A$ be a C*-algebra. Then $\mathrm{Aut}(A)$ is a topological group when
endowed with the topology of pointwise convergence. An action of a
topological group $G$ on $A$ is said to be \emph{continuous }if it is
continuous as a group homomorphism $G\rightarrow \mathrm{Aut}(A)$. \emph{In
the following, all the actions of topological groups are supposed to be
continuous}. The following is a natural example of a continuous action:

\begin{notation}
\label{nota:Lt} For a compact group $G$, let $C(G) $ be the commutative
C*-algebra of continuous complex-valued functions on $G$. We denote by $%
\mathtt{Lt}^{G}\colon G\rightarrow \mathrm{Aut}(C(G))$ the canonical action
by left translation, given by $\mathtt{Lt}_{g}^{G}(f)(h)=f(g^{-1}h)$ for $%
g,h\in G$ and $f\in C(G)$. When the group $G$ is clear from the context, we
write $\mathtt{Lt}$ instead of $\mathtt{Lt}^{G} $ to lighten the notation.
\end{notation}

We recall the definition of the Rokhlin property for compact group actions
on unital C*-algebras from \cite[Definition 3.2]{hirshberg_rokhlin_2007}.
The formulation given here is taken from \cite[Lemma~3.7]%
{gardella_rokhlin_2014}.

\begin{definition}
\label{Definition:Rokhlin} Let $G$ be a compact group, let $A$ be a unital
C*-algebra, and let $\alpha\colon G\to\mathrm{Aut}(A)$ be an action. We say
that $\alpha$ has the \emph{Rokhlin property} if for every $\varepsilon >0$,
for every finite subset $S\subseteq C( G) $, and every finite subset $%
F\subseteq A$, there exists a unital completely positive linear map $\psi
\colon C( G) \to A$ satisfying

\begin{itemize}
\item $\| ( \psi \circ \mathtt{Lt}_{g})(f)-(\alpha _{g}\circ \psi ) ( f)
\|<\varepsilon $ for all $f\in S$ and all $g\in G$;

\item $\| \psi ( f) a-a\psi (f) \| <\varepsilon $ for all $f\in S$ and all $%
a\in F$;

\item $\Vert \psi (f_{0}f_{1})-\psi (f_{0})\psi (f_{1})\Vert <\varepsilon $
for any $f,f_{0},f_{1}\in S$.
\end{itemize}
\end{definition}

\subsection{Direct and inverse limit constructions.\label{Sbs:DirInvLim}}

\begin{definition}
An\emph{\ inverse system of topological groups} is a family $(G_{i}, \pi
_{i,j})_{i,j\in I}$, where $I$ is an ordered set, $G_{i}$ are topological
groups, and $\pi _{i,j}\colon G_{j}\rightarrow G_{i}$, for $i\leq j$, is a
surjective continuous group homomorphism. Given such a countable inverse
system, we denote by $G=\varprojlim (G_{i}, \pi _{i,j})$ the inverse limit,
together with the canonical continuous surjective group homomorphisms $\pi
_{i, \infty }\colon G\rightarrow G_{i}$ for $i\in I$.

Similarly, a\emph{\ direct system of unital C*-algebras} is a family $%
(A_{i}, \iota _{i,j})_{i,j\in I}$, where $I$ is an ordered set, $A_{i}$ is a
unital C*-algebra, and $\iota _{i,j}\colon A_{i}\rightarrow A_{j}$, for $%
i\leq j$, is an injective unital *-homomorphism. We denote by $A=\varinjlim
(A_{i}, \iota _{i,j})$ the corresponding direct limit, together with the
canonical injective unital *-homomorphisms $\iota _{i, \infty }\colon
A_{i}\rightarrow A$ for $i\in I$.
\end{definition}

Next, we will see that one can construct actions of inverse limits of groups
on direct limits of C*-algebras in a natural way.

\begin{lemma}
\label{inverse limit action} Let $I$ be an ordered set, let $( A_{i,}\iota
_{i,j}) _{i,j\in I}$ be a direct system of unital C*-algebras with limit $A$%
, and let $( G_{i}, \pi _{i,j}) _{i,j\in I}$ be an inverse system of
topological groups with limit $G$. For every $i\in I$, let $\alpha
^{(i)}\colon G_{i}\to \mathrm{Aut}(A_{i})$ be an action satisfying 
\begin{equation}
\alpha _{g}^{(j)}\circ \iota _{i,j}=\iota _{i,j}\circ \alpha _{\pi _{i,j}(
g) }^{(i)}\text{\label{Equation:compatible}}
\end{equation}
for every $i,j\in I$ with $i\leq j$ and every $g\in G_{j}$. Then there
exists a unique action $\alpha \colon G\to\mathrm{Aut}(A)$ such that 
\begin{equation}
\alpha _{g}\circ \iota _{i, \infty }=\iota _{i, \infty }\circ \alpha _{\pi
_{i, \infty }( g) }^{(i)}\text{\label{Equation:limit-action}}
\end{equation}
for every $i\in I$ and $g\in G$.
\end{lemma}

\begin{proof}
It is clear that Equation (\ref{Equation:limit-action}) defines a unique
action of $G$ on $A$ in view of Equation (\ref{Equation:compatible}). We
check that such an action is continuous. For every $i\in I$, we identify $%
A_{i}$ with its image under $\iota _{i, \infty }$.

Fix $\varepsilon >0$, let $i\in I$ and let $F\subseteq A_i$ be a finite
subset. Since $\alpha ^{(i)}$ is continuous, there exists a neighborhood $U$
of the identity of $G_{i}$ such that $\|\alpha _{g}^{(i)}(x)-x\|<\varepsilon 
$ for every $x\in F$ and $g\in U$. Set $V=\pi_{i, \infty}^{-1}(U)$, which is
a neighborhood of the identity of $G$. For every $g\in V$ and $x\in F$, we
have 
\begin{equation*}
\left\Vert (\alpha _{g}^{(i)}\circ \iota _{i, \infty })(x)-\iota _{i, \infty
}(x)\right\Vert \leq \left\Vert \alpha _{\pi _{i, \infty }( g)
}^{(i)}(x)-x\right\Vert \leq \varepsilon \text{.}
\end{equation*}%
Since $i\in I$ is arbitrary and $\bigcup_{i\in I}\iota_{i,\infty}(A_i)$ is
dense in $A$, this concludes the proof.
\end{proof}

The definition of an amenable trace on a unital C*-algebra $A$ can be found
in \cite[Definition 6.2.1]{brown_c*-algebras_2008}. Observe that the set $%
\mathrm{T}_{\mathrm{am}}( A) $ of amenable traces on $A$ is a face of the
simplex $\mathrm{T}( A) $ of traces on $A$. Particularly, any extreme point
of $\mathrm{T}_{\mathrm{am}}( A) $ is also an extreme point of $\mathrm{T}(
A) $.\ Recall that every trace on a nuclear C*-algebra is amenable \cite[%
Theorem 4.2.1]{brown_invariant_2006}. The notion of locally reflexive
C*-algebra can be found in \cite[Definition 4.3.1]{brown_invariant_2006}.
Every exact C*-algebra is locally reflexive \cite[Corollary 9.4.1]%
{brown_c*-algebras_2008}. The following result is folklore, and we thank
Stuart White for suggesting this formulation.

\begin{lemma}
\label{lemma:WeakClosureR} Let $A$ be a separable, locally reflexive
C*-algebra, and let $\tau $ be a nonzero trace on $A$. Then $\overline{A}%
^{\tau }$ is isomorphic to the hyperfinite II$_{1}$-factor with separable
predual $R$ if and only if $\tau $ is amenable and extreme, and $A$ is
infinite dimensional.
\end{lemma}

\begin{proof}
It is well known that a trace is extreme if and only if $\overline{A}^{\tau
} $ is a factor (in which case it will be of type II$_{1}$ or I$_{n}$ for
some $n\in \mathbb{N}$). If $\tau $ is amenable, then $\overline{A}^{\tau }$
is hyperfinite by \cite[Corollary~4.3.4]{brown_invariant_2006}, because $A$
is locally reflexive. Finally, since $A$ is infinite dimensional, $\overline{%
A}^{\tau }$ must be isomorphic to $R$ by its uniqueness. Conversely, assume
that $\overline{A}^{\tau }\cong R$. Since the trace on $R$ is amenable, its
restriction to $A$, which agrees with $\tau $, must also be amenable.
Infinite dimensionality of $A$ is clear, so the proof is complete.
\end{proof}

\subsection{Subgroups with relative property (T)}

In this subsection, we recall the definition of relative property (T) for a
subgroup $\Delta $ of a discrete group $\Lambda $.

\begin{definition}
\label{df:RelpropT} Let $\Lambda $ be a discrete group and let $\Delta $ be
a subgroup. We say that $\Delta $ has \emph{relative property }(T) (of
Kazhdan--Margulis), if there exist a finite subset $F\subseteq \Lambda $ and 
$\varepsilon >0$ such that whenever $u\colon \Lambda \rightarrow U(H)$ is a
unitary representation of $\Lambda $ on a Hilbert space $H$, and $\xi \in H$
is a unit vector satisfying $\Vert u_{\gamma }(\xi )-\xi \Vert <\varepsilon $
for all $\gamma \in F$, then $H$ has a nonzero vector which is fixed by the
restriction of $u$ to $\Delta $.
\end{definition}

For $\Lambda =\Delta $, the definition above recovers the notion of property
(T) group. More generally, it is clear that if either $\Lambda $ \emph{or} $%
\Delta $ has property (T), then $\Delta \subseteq \Lambda $ has relative
property (T). There also exist inclusions of groups with relative property
(T), for which neither the subgroup nor the containing group have property
(T). One such example is $\mathbb{Z}^{2}\subseteq \mathbb{Z}^{2}\rtimes 
\mathrm{SL}_{2}(\mathbb{Z})$. (One can also replace SL$_{2}(\mathbb{Z})$
with any of its nonamenable subgroups, by a result of Burger.) Subgroups
with relative property (T) have been studied, among others, by Margulis~\cite%
{margulis_finitely-additive_1982}, Burger~\cite{burger_kazhdan_1991}, and
Jolissaint~\cite{jolissaint_property_2005}.

\section{Model action for profinite abelian groups\label{Section:model}}

\subsection{Profinite groups}

Let $\mathcal{C}$ be a class of groups closed under quotients, finite
products, and subgroups. A \emph{pro-}$\mathcal{C} $ group is a topological
group $G$ that can be realized as the inverse limit of groups from $\mathcal{%
C}$. Particularly, a group $G$ is said to be

\begin{itemize}
\item \emph{profinite }if it is pro-$\mathcal{C}$ for the class $\mathcal{C}$
of finite groups;

\item $\emph{pro}$-$p$ if it is a pro-$\mathcal{C}$ for the class $\mathcal{C%
}$ of finite $p$-groups.
\end{itemize}

It is clear that a profinite group is abelian if and only if it is pro-$%
\mathcal{C}$ for the class $\mathcal{C}$ of finite abelian groups.
Similarly, a pro-$p$ group is abelian if and only if it is pro-$\mathcal{C}$
for the class of finite abelian $p$-groups. Equivalent characterizations of
pro-$\mathcal{C}$ groups can be found in \cite[Theorem 2.1.3]%
{ribes_profinite_2010}. In particular, these characterizations show that a
topological group is profinite if and only if it is totally disconnected, if
and only if the identity of $G$ has a basis of neighborhoods made of open
subgroups \cite[Theorem 2.1.3]{ribes_profinite_2010}. Recall that, by \cite[%
Lemma 2.1.2]{ribes_profinite_2010}, a subgroup of a profinite abelian group
is open if and only if it is closed and has finite index.

Denote by $\mathcal{P}$ the set of prime numbers. A \emph{supernatural number%
} is a function $\boldsymbol{n}\colon \mathcal{P}\rightarrow \{0,1,2,\ldots
,\infty \}$. Recall also that a separable UHF-algebra is a unital C*-algebra
that is obtained as the direct limit of a countable direct system of full
matrix algebras. By a fundamental result of Glimm, any separable UHF-algebra
has the form $\otimes _{p\in \mathcal{P}}M_{p}^{\otimes \boldsymbol{n}(p)}$
for some supernatural number $\boldsymbol{n}$. The supernatural number can
be obtained intrinsically from the given separable UHF-algebra, and it is a
complete invariant for separable UHF-algebras up to *-isomorphism. Next, we
associate to each second countable profinite group, a canonical supernatural
number.

\begin{definition}
\label{Definition:supernatural} Let $G$ be a profinite group. The \emph{%
supernatural number associated with $G$} is defined by%
\begin{equation*}
\boldsymbol{n}_{G}(p)=\left\{ 
\begin{array}{cc}
\infty & \text{if }p\text{ divides the index of an open subgroup of }G\text{,%
} \\ 
0 & \text{otherwise.}%
\end{array}%
\right.
\end{equation*}%
We let $D_{G}$ be the UHF-algebra $\bigotimes_{p\in \mathcal{P}%
}M_{p}^{\otimes \boldsymbol{n}_{G}(p)}$ corresponding to $\boldsymbol{n}_{G}$%
.
\end{definition}

It is clear that $G$ is a pro-$p$ group if and only if $\boldsymbol{n}%
_{G}=p^{\infty }$ or, equivalently, $D_{G}=M_{p^{\infty}}$.

\subsection{Model action}

The goal of this subsection is to construct a model action $\delta ^{G}$ of $%
G$ on $D_{G}$ with the Rokhlin property; see Theorem~\ref%
{thm:ModelActionProfinite}. This action will be crucial in the next section,
where for certain nonamenable groups, we construct many non weakly cocycle
conjugate strongly outer actions on UHF-algebras.

\begin{remark}
\label{rmk:modelFiniteG} Suppose that $G$ is finite. Then $%
D_{G}=M_{|G|^{\infty }}$, and the model action in this case is rather easy
to describe. If $\lambda ^{G}\colon G\rightarrow U(\ell ^{2}(G))$ denotes
the left regular representation, then the model action $\delta ^{G}\colon
G\rightarrow \mathrm{Aut}(D_{G})$ is given by $\delta _{g}^{G}=\mathrm{Ad}%
(\lambda _{g}^{G})^{\otimes \mathbb{N}}$.
\end{remark}

The following is folklore; see, for example, \cite[Subsection 2.4]%
{izumi_finite_2004-1} (but note that the reference given there only proves
the statement about fixed point algebras for $G=\mathbb{Z}_p$). Since we
have not been able to find a reference, we include a short proof for the
convenience of the reader. (The proof given below uses the classification
results of \cite{izumi_finite_2004}, but a direct and elementary, although
longer, proof can also be given.)

\begin{lemma}
\label{lma:RokhlinFiniteAbelian} Let $G$ be a finite group. Then the action $%
\delta ^{G}\colon G\rightarrow \mathrm{Aut}(D_{G})$ described in the remark
above has the Rokhlin property. When $G$ is abelian, then $D_{G}\rtimes
_{\delta ^{G}}G$ is naturally isomorphic to $D_{\widehat{G}}$, in such a way
that the dual action $\widehat{\delta ^{G}}\colon \widehat{G}\rightarrow 
\mathrm{Aut}(D_{G}\rtimes _{\delta ^{G}}G)$ is conjugate to $\delta ^{%
\widehat{G}}$.
\end{lemma}

\begin{proof}
It is easy to see that $\delta ^{G}$ has the Rokhlin property, since there
is a unital and equivariant embedding $C(G)\rightarrow B(\ell ^{2}(G))\cong
M_{|G|}$ as multiplication operators. The crossed product $D_{G}\rtimes
_{\delta ^{G}}G$ is a UHF-algebra by \cite[Corollary~3.11]%
{gardella_crossed_2014}. Since there are unital inclusions 
\begin{equation*}
D_{G}\subseteq D_{G}\rtimes _{\delta ^{G}}G\subseteq \mathcal{K}%
(\ell^2(G))\otimes D_{G}\cong D_{G},
\end{equation*}%
it follows that $D_{G}$ and $D_{G}\rtimes _{\delta ^{G}}G$ have the same
corresponding supernatural number, and hence they are isomorphic. In
particular, $D_{G}\rtimes _{\delta ^{G}}G$ is isomorphic to $D_{\widehat{G}}$%
.

It remains to identify the dual action of $\delta^{G}$. Observe that $%
\delta^G$ is approximately representable in the sense of \cite[Definition~3.6%
]{izumi_finite_2004}, as one may take the unitaries $u(g)$ appearing in said
definition to be $u(g)=\lambda_g^{\otimes n}$ for a large enough $n\in%
\mathbb{N}$. By \cite[part~(2) of Lemma~3.8]{izumi_finite_2004}, it follows
that the dual action $\widehat{\delta^{G}}$ has the Rokhlin property as an
action of $\widehat{G}$ on $D_G\rtimes_{\delta^G}G\cong D_{\widehat{G}}$.
Since for $\chi\in\widehat{G}$, the automorphisms $\delta^{\widehat{G}%
}_{\chi}$ and $\widehat{\delta^{G}_{\chi}}$ are both approximately inner, it
follows from \cite[Theorem~3.5]{izumi_finite_2004} that $\widehat{\delta ^{G}%
}$ is conjugate to $\delta ^{\widehat{G}}$, and the proof is finished.
\end{proof}

The model action $\delta^G$ for an arbitrary profinite abelian group $G$
will be constructed from its finite quotients using Lemma~\ref%
{lma:RokhlinFiniteAbelian}. The following proposition is the inductive step
in the construction.

\begin{proposition}
\label{prop:IndStep} Let $H$ be a finite abelian group, let $N$ be a
subgroup of $H$, and set $Q=H/N$ with quotient map $\pi\colon H\to Q$.
Denote by $\delta^H\colon H\to\mathrm{Aut}(D_H)$ and $\delta^Q\colon Q\to%
\mathrm{Aut}(D_Q)$ the actions described in Remark~\ref{rmk:modelFiniteG}.
Then there is an injective unital homomorphism $\iota\colon D_Q\to D_H$
satisfying 
\begin{equation*}
\delta^H_h\circ\iota=\iota\circ\delta^Q_{\pi(h)}
\end{equation*}
for all $h\in H$.
\end{proposition}

\begin{proof}
For a finite group $K$, we denote by $\{\xi _{k}^{K}\}_{k\in K}$ the
canonical basis of $\ell ^{2}(K)$. Also, when $K$ is abelian, we write $%
\widehat{K}$ for its dual group, and an element of $\widehat{K}$ will be
denoted, with a slight abuse of notation, by $\widehat{k}$. Observe that $%
\widehat{Q}$ is a subgroup of $\widehat{H}$, and that $\widehat{H}/\widehat{Q%
}\cong \widehat{N}$. Fix a section $s\colon \widehat{N}\rightarrow \widehat{H%
}$. Then $s$ induces a unitary $U\colon \ell ^{2}(\widehat{Q})\otimes \ell
^{2}(\widehat{N})\rightarrow \ell ^{2}(\widehat{H})$ given by 
\begin{equation*}
U(\xi _{\widehat{q}}^{\widehat{Q}}\otimes \xi _{\widehat{n}}^{\widehat{N}%
})=\xi _{\widehat{q}s(\widehat{n})}^{\widehat{H}}
\end{equation*}%
for every $\widehat{q}\in \widehat{Q}$ and every $\widehat{n}\in \widehat{N}$%
. Define a unital embedding $\varphi \colon B(\ell ^{2}(\widehat{Q}%
))\rightarrow B(\ell ^{2}(\widehat{H}))$ by 
\begin{equation*}
\varphi (a)(\xi _{\widehat{q}s(\widehat{n})}^{\widehat{H}})=U(a(\xi _{%
\widehat{q}}^{\widehat{Q}})\otimes \xi _{\widehat{n}}^{\widehat{N}})
\end{equation*}%
for every $a\in B(\ell ^{2}(\widehat{Q}))$, for every $\widehat{q}\in 
\widehat{Q}$ and every $\widehat{n}\in \widehat{N}$. Let $\widehat{q}\in 
\widehat{Q}$. We claim that $\varphi (\lambda _{\widehat{q}}^{\widehat{Q}%
})=\lambda _{\widehat{q}}^{\widehat{H}}$. To see this, let $\widehat{p}\in 
\widehat{Q}$ and let $\widehat{n}\in \widehat{N}$. Then 
\begin{equation*}
\varphi (\lambda _{\widehat{q}}^{\widehat{Q}})(\xi _{\widehat{p}s(\widehat{n}%
)}^{\widehat{H}})=U(\lambda _{\widehat{q}}^{\widehat{Q}}(\xi _{\widehat{p}}^{%
\widehat{Q}})\otimes \xi _{\widehat{n}}^{\widehat{N}})=U(\xi _{\widehat{q}%
\widehat{p}}^{\widehat{Q}}\otimes \xi _{\widehat{n}}^{\widehat{N}})=\xi _{%
\widehat{q}\widehat{p}s(\widehat{n})}^{\widehat{H}}=\lambda _{\widehat{q}}^{%
\widehat{H}}(\xi _{\widehat{p}s(\widehat{n})}^{\widehat{H}}).
\end{equation*}%
This proves the claim. It follows that $\varphi $ induces, upon taking its
infinite tensor product, a unital injective homomorphism $\psi \colon D_{%
\widehat{Q}}\rightarrow D_{\widehat{H}}$, which moreover satisfies 
\begin{equation*}
\psi \circ \delta _{\widehat{q}}^{\widehat{Q}}=\delta _{\widehat{q}}^{%
\widehat{H}}\circ \psi
\end{equation*}%
for all $\widehat{q}\in \widehat{Q}$, by the claim above. After taking
crossed products by $\widehat{Q}$ and $\widehat{H}$, and using Lemma~\ref%
{lma:RokhlinFiniteAbelian}, we obtain a unital embedding $\iota \colon
D_{Q}\rightarrow D_{H}$, which satisfies $\delta _{h}^{H}\circ \iota =\iota
\circ \delta _{\pi (h)}^{Q}$ for all $h\in H$. This completes the proof.
\end{proof}

Here is the main result of this section.

\begin{theorem}
\label{thm:ModelActionProfinite} Let $G$ be a second countable, abelian,
profinite group. Let $D_{G}$ denote the UHF-algebra associated with $G$ as
in Definition~\ref{Definition:supernatural}. Then there exists a canonical
action $\delta ^{G}\colon G\rightarrow \mathrm{Aut}(D_{G})$ with the
following properties.

\begin{enumerate}
\item There exists an equivariant unital embedding $(C(G), \mathtt{Lt}^G)\to
(D_G, \delta^G)$.

\item $(D_G^{\otimes\mathbb{N}},(\delta^G)^{\otimes\mathbb{N}})$ is
equivariantly isomorphic to $(D_G, \delta^G)$.

\item $\delta^G$ has the Rokhlin property.

\item The fixed point algebra $D_{G}^{\delta ^{G}}$ is isomorphic to $D_{G}$.
\end{enumerate}

Moreover, $\delta ^{G}$ is---up to conjugacy---the \emph{unique} action of $%
G $ on $D_{G}$ with the Rokhlin property. Furthermore, if $E$ is a unital
C*-algebra with $D_{G}\otimes E\cong E$ and $\beta \colon G\rightarrow 
\mathrm{Aut}(E)$ is an action with the Rokhlin property, then $\beta $ is
conjugate to $\delta ^{G}\otimes \beta $.
\end{theorem}

\begin{proof}
Since the group $G$ is fixed, we drop the subscript $G$ from all algebras
and actions, in order to lighten the notation. We first construct the
action, and then show that it has the desired properties. Let $\mathcal{V}$
be the collection of open subgroups of $G$, and observe that $\mathcal{V}$
is countable. Define an inverse system $(G_{i},\pi _{i,j})_{i,j\in I}$ of
finite groups as follows. Set $I=\mathcal{V}$ ordered by \emph{reverse
inclusion}. For $i\in I$, let $G_{i}=G/i$, and for $i,j\in I$ with $i\leq j$%
, let $\pi _{i,j}\colon G_{j}\rightarrow G_{i}$ be the canonical quotient
map. Then $G=\varinjlim (G_{i},\pi _{i,j})$. By Proposition~\ref%
{prop:IndStep}, for every $i,j\in I$ with $i\leq j$, there exists a unital
embedding $\iota _{i,j}\colon D_{G_{i}}\rightarrow D_{G_{j}}$ satisfying $%
\delta _{g}^{G_{j}}\circ \iota _{i,j}=\iota _{i,j}\circ \delta _{\pi
_{i,j}(g)}^{G_{i}}$ for all $g\in G_{j}$. Observe that $D$ can be identified
with the direct limit of the UHF-algebras $D_{G_{i}}$, for $i\in I$, with
connective maps $\iota _{i,j}$ for $i,j\in I$ with $i\leq j$. By Lemma~\ref%
{inverse limit action}, there exists an induced action $\delta \colon
G\rightarrow \mathrm{Aut}(D)$ given by 
\begin{equation*}
\delta _{g}(\iota _{i,\infty }(a))=\iota _{i,\infty }(\delta _{\pi
_{i,\infty }(g)}^{G_{i}}(a))
\end{equation*}%
for all $g\in G$, for all $i\in I$, and all $a\in D_{G_{i}}$.

(1): Let $i\in I$. Observe that the restriction of $\delta ^{G_{i}}$ to $%
\mathbb{C}\rtimes \widehat{G_{i}}\cong C(G_{i})$ is naturally conjugate to
the left translation action $\mathtt{Lt}^{G_{i}}$. In particular, there is a
unital equivariant embedding $\phi _{i}\colon (C(G_{i}),\mathtt{Lt}%
^{G_{i}})\rightarrow (D_{G_{i}},\delta ^{G_{i}})$. For $i,j\in I$ with $%
i\leq j$, denote by $\pi _{i,j}^{\ast }\colon C(G_{i})\rightarrow C(G_{j})$
the injective unital *-homomorphism given by $\pi _{i,j}^{\ast }(f)=f\circ
\pi _{i,j}$ for all $f\in C(G_{i})$. Then the maps $\phi _{i}$ are easily
seen to satisfy $\iota _{i,j}\circ \phi _{i}=\phi _{j}\circ \pi _{i,j}^{\ast
}$ for all $i,j\in I$ with $i\leq j$. By the universal property of direct
limits, it follows that there exists a unital equivariant embedding $(C(G),%
\mathtt{Lt})\rightarrow (D,\delta )$, as desired.

(2): This is an easy consequence of the fact that $\delta ^{G_{i}}$ is
conjugate to $(\delta ^{G_{i}})^{\otimes \mathbb{N}}$ for every $i\in I$.

(3): This is an immediate consequence of (1) and (2).

(4): By part~(1) of Corollary~3.11 in \cite{gardella_crossed_2014}, the
fixed point algebra $D^{\delta }$ is a UHF-algebra, and it absorbs $D$ by 
\cite[Theorem~4.3]{gardella_crossed_2014}. Since it is obviously unitally
embedded in $D$, it follows from \cite[Proposition~5.12]{toms_strongly_2007}
that $D^{\delta }$ is isomorphic to $D$. The last part of the theorem is a
consequence of \cite[Theorem~X.4.5]{gardella_compact_2015}.
\end{proof}

\section{Uncountably many actions\label{Sct:uncountably}}

In this section, given a countable group $\Lambda $ containing an infinite
subgroup $\Delta $ with relative property (T)---which we fix once and for
all---and given a UHF-algebra $A$ of infinite type, we construct uncountably
many strongly outer actions of $\Lambda $ on $A$, which are not weakly
cocycle conjugate; see Theorem~\ref{Theorem:uncountably}. In fact, we
perform the construction for an arbitrary separable unital C*-algebra $A$
satisfying the following properties: $A$ is locally reflexive, $M_{p^{\infty
}}$-absorbing for some prime $p$, has an amenable trace, and is isomorphic
to its infinite tensor product $A^{\otimes \mathbb{N}}$.

Let $G$ be a second countable abelian pro-$p$ group, and let $\delta
^{G}\colon G\rightarrow \mathrm{Aut}(D_{G})$ be the action constructed in
Theorem~\ref{thm:ModelActionProfinite}. We denote in the same fashion the
extension of $\delta ^{G}$ the weak closure of $D_{G}$ with respect to its
unique trace, which can be regarded as an action on the hyperfinite II$_{1}$
factor $R$. In the following lemma, we will use the pairing function from
Definition~\ref{Definition:mTheta-relative}. We write $\Gamma $ for the
Pontryagin dual of $G$, and we let $m^{\Gamma }\colon \Gamma \times \Gamma
\rightarrow \Gamma $ be the multiplication operation. Recall also that $%
\mathtt{Lt}\colon G\rightarrow \mathrm{Aut}(C(G))$ denotes the action by
left translation.

\begin{lemma}
\label{Lemma:key}Let $N$ be the algebra $\left( R\otimes R\right) ^{%
\overline{\otimes }\Lambda }$, and let $\rho $ be the action $\left( \delta
^{G}\otimes \mathrm{id}_{R}\right) ^{\otimes \Lambda }$ of $G$ on $N$.
Define $B$ to be the fixed point algebra $N^{\rho }$ of $\rho $, which is
isomorphic to the hyperfinite II$_{1}$ factor by Theorem \ref%
{thm:ModelActionProfinite}. Consider the Bernoulli $\left( \Lambda
\curvearrowright \Lambda \right) $-action $\beta $ with base $R\otimes R$,
which is an action on $N$, and its restriction $\alpha $ to $B$. Then there
exist bijections $\eta \colon H_{\Delta ,w}^{1}(\alpha )\rightarrow \Gamma $
and $\eta ^{(2)}\colon H_{\Delta ,w}^{1}(\alpha \otimes \alpha )\rightarrow
\Gamma $ satisfying 
\begin{equation*}
\eta ^{(2)}\circ m_{\Delta }^{\alpha }=m^{\Gamma }\circ (\eta \times \eta ).
\end{equation*}
\end{lemma}

\begin{proof}
Let $\zeta $ be the restriction of $\beta $ to $\Delta $. Let $u\colon
\Lambda \rightarrow U(B)$ be a weak $1$-cocycle for $\alpha $. Since $\alpha 
$ is the restriction of $\beta $ to $B$, we deduce that $u$ is also a $1$%
-cocycle for $\beta $. It is shown in \cite[Section 3]{popa_some_2006} that
the von Neumann-algebraic Bernoulli $(\Lambda \curvearrowright \Lambda )$%
-action $\beta $ satisfies the assumptions of \cite[Theorem 4.1]%
{popa_some_2006}. Applying \cite[Theorem 4.1]{popa_some_2006} in the case of
weak $1$-cocycles, with $S=S_{1}=\{1\}$ in the notation of \cite%
{popa_some_2006}, we conclude that $u$ is a $\Delta $-local weak coboundary
for $\beta $. Therefore, there exist a unitary $v\in U(N)$ and a function $%
\mu \colon \Delta \rightarrow \mathbb{T}$ satisfying $u_{\gamma }=\mu
_{\gamma }v^{\ast }\beta _{\gamma }(v)$ for every $\gamma \in \Delta $. Fix $%
g\in G$. Applying $\rho _{g}$ to the previous identity, and using that $%
u_{\gamma }\in N^{\rho }=B$, yields 
\begin{equation*}
\mu _{\gamma }v^{\ast }\beta _{\gamma }(v)=u_{\gamma }=\rho _{g}(u_{\gamma
})=\mu _{\gamma }\rho _{g}(v)^{\ast }\beta _{\gamma }(\rho _{g}(v))
\end{equation*}%
for every $\gamma \in \Delta $. Hence $\rho _{g}(v)v^{\ast }$ is fixed by $%
\zeta $. By Proposition \ref{Proposition:bernoulli-wm}, $\zeta $ is mixing.
Therefore by Remark~\ref{rmk:CharMixing} we conclude that $\rho
_{g}(v)v^{\ast }$ is a scalar. Define a function $\chi _{u}\colon
G\rightarrow \mathbb{T}$ by $\chi _{u}(g)=\rho _{g}(v)v^{\ast }$ for $g\in G$%
.

\begin{claim*}
$\chi _{u}$ is well-defined (that is, independent of the choice of $\mu $
and $v$).
\end{claim*}

\begin{proof}[Proof of claim]
Fix $v,v^{\prime }\in U(N)$ and $\mu ,\mu ^{\prime }\colon \Delta
\rightarrow \mathbb{T}$ satisfying 
\begin{equation*}
u_{\gamma }=\mu _{\gamma }v^{\ast }\beta _{\gamma }(v)=\mu _{\gamma
}^{\prime }(v^{\prime })^{\ast }\beta _{\gamma }(v^{\prime })
\end{equation*}%
for every $\gamma \in \Delta $. We want to show that $\rho _{g}(v)v^{\ast
}=\rho _{g}(v^{\prime })(v^{\prime })^{\ast }$ for all $g\in G$. The above
identity implies that 
\begin{equation*}
\beta _{\gamma }(v^{\prime }v^{\ast })=\mu _{\gamma }\overline{\mu }_{\gamma
}^{\prime }v^{\prime }v^{\ast }
\end{equation*}%
for all $\gamma \in \Delta $. In particular, the 1-dimensional subspace of $%
L^{2}(N)$ spanned by $v^{\prime }v^{\ast }$ is invariant by $\zeta $. Hence,
it follows from Remark~\ref{rmk:CharMixing} and the fact that $\zeta $ is
mixing that $v^{\prime }v^{\ast }$ is a scalar, which we abbreviate to $z\in 
\mathbb{T}$. Thus, 
\begin{equation*}
\rho _{g}(v^{\prime })(v^{\prime })^{\ast }=\rho _{g}(zv)(zv)^{\ast }=z%
\overline{z}\rho _{g}(v)v^{\ast }=\rho _{g}(v)v^{\ast }
\end{equation*}%
for all $g\in G$, as desired.
\end{proof}

\begin{claim*}
$\chi _{u}$ is a character on $G$
\end{claim*}

\begin{proof}[Proof of claim]
First, observe that $\chi _{u}$ is a continuous function, since $\rho _{g}$
is a continuous action. To check the character condition, let $g,h\in G$.
Then 
\begin{equation*}
\chi _{u}(gh)=\rho _{gh}(v)v^{\ast }=\rho _{g}(\rho _{h}(v)v^{\ast })\rho
_{g}(v)v^{\ast }=\chi _{u}(g)\chi _{u}(h)\text{,}
\end{equation*}%
so the claim is proved.
\end{proof}

\begin{claim*}
For $u\in Z_{w}^{1}(\alpha )$, the character $\chi _{u}$ only depends on the 
$\Delta $-local weak cohomology class of $u$.
\end{claim*}

\begin{proof}[Proof of claim]
Let $u^{\prime }\in Z_{w}^{1}(\alpha )$ be $\Delta $-locally weakly
cohomologous to $u$, and let $w\in U(N)$ satisfy $u_{\gamma }^{\prime
}=w^{\ast }u_{\gamma }\alpha _{\gamma }(w)\mathrm{mod}\mathbb{C}$ for every $%
\gamma \in \Delta $. Let $v\in U(N)$ be an eigenvector for $\rho $ with
eigenvalue $\chi _{u}$, such that $u_{\gamma }=v^{\ast }\overline{\beta }%
_{\gamma }\left( v\right) \mathrm{mod}\mathbb{C}$ for every $\gamma \in
\Delta $. Then 
\begin{equation*}
u_{\gamma }^{\prime }=\mu _{\gamma }(vw)^{\ast }\beta _{\gamma }(vw)\mathrm{%
mod}\mathbb{C}
\end{equation*}%
for every $\gamma \in \Delta $, and hence $vw\in U(N)$ is an eigenvector for 
$\rho $ with eigenvalue $\chi _{u}$. Therefore $\chi _{u^{\prime }}=\chi
_{u} $.
\end{proof}

In view of the previous claims, we can define a function $\eta \colon
H_{\Delta ,w}^{1}(\alpha )\rightarrow \Gamma $ by $\eta ([u])=\chi _{u}$ for
all $[u]\in H_{\Delta ,w}^{1}(\alpha )$.

\begin{claim*}
The map $\eta \colon H_{\Delta ,w}^{1}(\alpha )\rightarrow \Gamma $ is
surjective.
\end{claim*}

\begin{proof}[Proof of claim]
Fix $\omega \in \Gamma $. Since $\omega $ is a continuous function $\omega
\colon G\rightarrow \mathbb{C}$, we can regard $\omega $ as a (unitary)
element in $C(G)$. Observe that $\omega $ is an eigenvector for $\mathtt{Lt}$
with eigenvalue $\omega $. By part~(1) of Theorem~\ref%
{thm:ModelActionProfinite}, there exists an equivariant unital embedding $%
(C(G),\mathtt{Lt})\rightarrow (D,\delta )$. Furthermore, there exists an
equivariant unital embedding 
\begin{equation*}
(D,\delta )\rightarrow \left( (D\otimes A)^{\otimes \Lambda },(\delta
\otimes \mathrm{id}_{A})^{\otimes \Lambda }\right) \text{.}
\end{equation*}%
Composing these maps, one can conclude that there exists an equivariant
unital embedding 
\begin{equation*}
(C(G),\mathtt{Lt})\rightarrow \left( (D\otimes A)^{\otimes \Lambda },(\delta
\otimes \mathrm{id}_{A})^{\otimes \Lambda }\right) \text{.}
\end{equation*}%
Identifying $C(G)$ with its image inside $(D\otimes A)^{\otimes \Lambda }$,
we can regard $\omega $ as an element of $(D\otimes A)^{\otimes \Lambda }$,
which is an eigenvector for $(\delta \otimes \mathrm{id}_{A})^{\otimes
\Lambda }$ with eigenvalue $\omega $. In turn, this gives an element $v$ of
the weak closure $N$ of $(D\otimes A)^{\otimes \Lambda }$ which is an
eigenvector for $\rho $ with eigenvalue $\omega $. Define a function $%
u\colon \Lambda \rightarrow N$ by $u_{\gamma }=v^{\ast }\beta _{\gamma }(v)$
for all $\gamma \in \Lambda $. For every $g\in G$, we have%
\begin{equation*}
\rho _{g}(u_{\gamma })=\rho _{g}(v)^{\ast }\beta _{\gamma }(\rho
_{g}(v))=v^{\ast }\beta _{\gamma }(v)=u_{\gamma }
\end{equation*}%
for all $\gamma \in \Lambda $. It follows that $u$ takes values in $%
B=N^{\rho }$. On the other hand, given $\gamma ,\lambda \in \Lambda $, we
have 
\begin{equation*}
u_{\gamma }\alpha _{\gamma }(u_{\lambda })=v^{\ast }\beta _{\gamma
}(v)\alpha _{\gamma }(v^{\ast }\beta _{\lambda }(v))=v^{\ast }\beta _{\gamma
\lambda }(v)=u_{\gamma \lambda }{}\mathrm{mod}\mathbb{C}\text{.}
\end{equation*}%
Therefore $u$ is a weak $1$-cocycle for $\alpha $, and $\chi _{u}=\omega $.
It follows that $\eta $ is surjective, as desired.
\end{proof}

\begin{claim*}
The map $\eta \colon H_{\Delta ,w}^{1}(\alpha )\rightarrow \Gamma $ is
injective (and hence a bijection).
\end{claim*}

\begin{proof}[Proof of claim]
Let $u_{0},u_{1}\in Z_{w}^{1}(\alpha )$ satisfy $\chi _{u_{0}}=\chi _{u_{1}}$%
. Denote by $\omega $ this character. Find eigenvectors $v_{0},v_{1}\in U(B)$
for $\rho $ with eigenvalue $\omega $ such that $u_{j,\gamma }=v_{j}^{\ast
}\beta _{\gamma }(v_{j}){}\mathrm{mod}\mathbb{C}$ for all $\gamma \in \Delta 
$ and $j=0,1$. Set $w=v_{0}^{\ast }v_{1}$, which is a unitary in $B$. For
every $\gamma \in \Delta $, we have%
\begin{equation*}
w^{\ast }u_{0,\gamma }\alpha _{\gamma }(w)=(v_{0}^{\ast }v_{1})^{\ast
}u_{0,\gamma }\alpha _{\gamma }(v_{0}^{\ast }v_{1})=v_{1}^{\ast
}v_{0}(v_{0}^{\ast }\beta _{\gamma }(v_{0}))\beta _{\gamma }(v_{0}^{\ast
}v_{1})=v_{1}^{\ast }\beta _{\gamma }(v_{1})=u_{1,\gamma }\ \mathrm{mod}%
\mathbb{C}\text{.}
\end{equation*}%
Therefore $w$ witnesses the fact that $u_{0}$ and $u_{1}$ are $\Delta $%
-locally weakly cohomologous. Thus $[u_{0}]=[u_{1}]\in H_{\Delta
,w}^{1}\left( \overline{\alpha }\right) $, and $\eta $ is injective.
\end{proof}

We now turn to the construction of the map $\eta ^{(2)}\colon H_{\Delta
,w}^{1}(\alpha \otimes \alpha )\rightarrow \Gamma $. Observe that $\alpha
\otimes \alpha $ is conjugate to $\alpha $. Let $u\in Z_{w}^{1}(\alpha
\otimes \alpha )$, and choose a unitary $v\in U(N\overline{\otimes }N)$
satisfying $u_{\gamma }=v^{\ast }(\beta _{\gamma }\otimes \beta _{\gamma
})(v){}\mathrm{mod}\mathbb{C}$ for all $\gamma \in \Delta $. As before, one
checks that $v$ is an eigenvector for $\rho \otimes \rho $, and that its
eigenvalue $\kappa _{u}$ is a character in $\Gamma $, which is independent
of $v$. Similarly to what was done above, one defines the map $\eta
^{(2)}\colon H_{\Delta ,w}^{1}(\alpha \otimes \alpha )\rightarrow \Gamma $
by $\eta ^{(2)}([u])=\kappa _{u}$ for all $[u]\in H_{\Delta ,w}^{1}(\alpha
\otimes \alpha )$.

It remains to prove the identity $\eta ^{(2)}\circ m_{\Delta }^{\alpha
}=m^{\Gamma }\circ (\eta \times \eta )$. Let $[u],[u^{\prime }]\in H_{\Delta
,w}^{1}(\alpha )$, and set $\omega =\eta ([u])$ and $\omega ^{\prime }=\eta
([u^{\prime }])$. Find eigenvectors $v,v^{\prime }\in U(B)$ for $\rho $ with
eigenvalues $\omega $ and $\omega ^{\prime }$, respectively, satisfying 
\begin{equation*}
u_{\gamma }=v^{\ast }\beta _{\gamma }(v){}\mathrm{mod}\mathbb{C}\ \ 
\mbox{
and }\ \ u_{\gamma }^{\prime }=v^{\prime \ast }\beta _{\gamma }(v^{\prime
}){}\mathrm{mod}\mathbb{C}
\end{equation*}%
for all $\gamma \in \Delta $. Hence $(u\otimes u^{\prime })_{\gamma
}=(v\otimes v^{\prime })^{\ast }(\overline{\beta }_{\gamma }\otimes 
\overline{\beta }_{\gamma })(v\otimes v^{\prime }){}\mathrm{mod}\mathbb{C}$
for every $\gamma \in \Delta $. Since $v\otimes v^{\prime }$ is an
eigenvector for $\rho \otimes \rho $ with eigenvalue $\omega \omega ^{\prime
}$, this shows that 
\begin{equation*}
(\eta ^{(2)}\circ m_{\Delta }^{\alpha })([u],[u^{\prime }])=\omega \omega
^{\prime }=\eta ([u])\eta ([u^{\prime }])=(m^{\Gamma }\circ (\eta \times
\eta ))([u],[u^{\prime }])\text{.}
\end{equation*}%
This concludes the proof of the lemma.
\end{proof}

We fix now a C*-algebra $A$ which is locally reflexive, $M_{p^{\infty }}$%
-absorbing for some prime $p$, has an amenable trace, and is isomorphic to
its infinite tensor product $A^{\otimes \mathbb{N}}$. We also fix a prime $p$
such that $A\cong A\otimes M_{p^{\infty }}$. We will frequently use the
notation for Bernoulli actions from Notation~\ref{Notation:Bernoulli}. We
write $D$ for $M_{p^{\infty }}$. We also fix an isomorphism $\phi \colon
A\rightarrow A^{\otimes \Lambda }$. Using this isomorphism, we let $\sigma
\colon \Lambda \rightarrow \mathrm{Aut}(A)$ denote the action given by 
\begin{equation*}
\sigma _{\gamma }=\phi ^{-1}\circ (\beta _{\Lambda \curvearrowright \Lambda
,A})_{\gamma }\circ \phi
\end{equation*}%
for all $\gamma \in \Lambda $.

Consider the diagonal action $(\delta ^{G})^{\otimes \Lambda }\colon
G\rightarrow \mathrm{Aut}(D^{\otimes \Lambda })$, and denote by $E_{G}$ its
fixed point algebra, which, by parts~(2) and~(4) of Theorem~\ref%
{thm:ModelActionProfinite}, is isomorphic to $D$. Since $(\delta
^{G})^{\otimes \Lambda }$ and $\beta _{\Lambda \curvearrowright \Lambda ,D}$
commute, $\beta _{\Lambda \curvearrowright \Lambda ,D}$ restricts to an
action $\beta _{\Lambda \curvearrowright \Lambda ,D}|_{E_{G}}\colon \Lambda
\rightarrow \mathrm{Aut}(E_{G})$.

\begin{definition}
\label{df:alpha} For each pro-$p$ group $G$, we choose an isomorphism $\xi
_{G}\colon A\rightarrow E_{G}\otimes A$. Now, we define an action $\alpha
^{G}\colon \Lambda \rightarrow \mathrm{Aut}(A)$ by 
\begin{equation*}
\alpha _{\gamma }^{G}=\xi _{G}^{-1}\circ (\beta _{\Lambda \curvearrowright
\Lambda ,D}|_{E_{G}}\otimes \sigma )_{\gamma }\circ \xi _{G}
\end{equation*}%
for all $\gamma \in \Lambda $.
\end{definition}

It will be shown in Theorem~\ref{Theorem:conjugacy} that for non-isomorphic
pro-$p$ groups $G_0$ and $G_1$, the actions $\alpha^{G_0}$ and $\alpha^{G_1}$
are not weakly cocycle conjugate. In order to do this, we will need to study
the weak extensions of these actions with respect to certain invariant
traces. Our next result provides us with a canonical subset of $T(A)$
consisting of traces that are $\alpha^{G}$-invariant for every pro-$p$ group 
$G$. Later, in Proposition~\ref{prop:AlternativeDescrAlpha}, we will show
that for any of these traces and for any $G$, the weak extension of any of $%
\alpha^G$ is mixing. For a pro-$p$ group $G$, we denote by $\tau_{E_G}$ the
(unique) trace on $E_G$.

\begin{proposition}
\label{prop:InvariantTracesAllG} Adopt the notation from the previous
discussion, and define a continuous, affine map $\iota\colon T(A)\to T(A)$
by $\iota(\tau)=\tau^{\otimes\Lambda}\circ\phi$ for all $\tau\in T(A)$. If $%
\tau$ is extreme and amenable, then so is $\iota(\tau)$.

Moreover, if $G$ is any pro-$p$ group, then $\iota(\tau)=(\tau_{E_G}\otimes
\iota(\tau))\circ\xi_G$ for all $\tau\in T(A)$. In particular, $\iota(\tau)$
is $\alpha^{G}$-invariant for all $\tau\in T(A)$.
\end{proposition}

\begin{proof}
The first assertion is standard (and in our case, it follows from Lemma~\ref%
{lemma:WeakClosureR}, since the weak closure of $A^{\otimes\Lambda}$ with
respect to $\tau^{\otimes\Lambda}$ is canonically isomorphic to $(\overline{A%
}^{\tau})^{\otimes\Lambda}\cong R^{\otimes\Lambda}\cong R$.)

Let $G$ be a pro-$p$ group and let $\tau\in T(A)$. Observe that $%
(\tau_{E_G}\otimes \iota(\tau))\circ\xi_G$ is an $\alpha^G$-invariant trace
on $A$. Hence, it suffices to show that this trace equals $\iota(\tau)$.

Observe that since $E_G$ is a UHF-algebra of infinite type, the isomorphism $%
\xi_G\colon A\to E_G\otimes A$ is approximately unitarily equivalent to the
second tensor factor embedding $\kappa\colon A \to E_G\otimes A$ given by $%
\kappa(a)=1_{E_G}\otimes a$ for all $a\in A$; see \cite[Corollary~1.12]%
{toms_strongly_2007}. It follows that $\xi_G$ and $\kappa$ induce the same
map at the level of traces. Using this at the first step in the following
computation, we conclude that 
\begin{equation*}
(\tau_{E_G}\otimes \iota(\tau))\circ\xi_G= (\tau_{E_G}\otimes
\iota(\tau))\circ\kappa= \iota(\tau)
\end{equation*}
for all traces $\tau\in T(A)$, as desired.
\end{proof}

Our next goal is to establish a number of properties for $\alpha ^{G}$; this
will be done in Proposition~\ref{prop:AlternativeDescrAlpha}. In order to do
this, we need an alternative description of $\alpha ^{G}$. Since the group $%
G $ will be fixed from now on and until Theorem~\ref{Theorem:conjugacy}, we
will drop it from the notation for the actions $\delta ^{G}$ and $\alpha
^{G} $, as well as from the notation for the algebra $E_{G}$. In Theorem~\ref%
{Theorem:conjugacy}, we will show that for nonisomorphic $G_{0}$ and $G_{1}$%
, the actions constructed above are not weakly cocycle conjugate. Until
then, we will work with a fixed pro-$p$ group $G$.

Observe that the Bernoulli action $\beta _{\Lambda \curvearrowright \Lambda
,D\otimes A}$ commutes with the diagonal action 
\begin{equation*}
(\delta \otimes \mathrm{id}_{A})^{\otimes \Lambda }\colon G\rightarrow 
\mathrm{Aut}((D\otimes A)^{\otimes \Lambda }).
\end{equation*}%
Thus, with $B$ denoting the fixed point algebra of $(\delta \otimes \mathrm{%
id}_{A})^{\otimes \Lambda }$, the action $\beta _{\Lambda \curvearrowright
\Lambda ,D\otimes A}$ restricts to an action $\widetilde{\alpha }\colon
\Lambda \rightarrow \mathrm{Aut}(B)$.

Fix an amenable extreme trace $\tau _{0}$ on $D\otimes A$ and let $%
\widetilde{\tau }$ be the trace $\tau _{0}^{\otimes \Lambda }$ on $\left(
D\otimes A\right) ^{\otimes \Lambda }$. Then $\overline{D\otimes A}^{\tau
_{0}}$ is isomorphic to $R$ by Lemma~\ref{lemma:WeakClosureR}, and the
extension of $\tau _{0}$ to $\overline{D\otimes A}^{\tau _{0}}$ is the
unique trace on $R$. We identify $\overline{\beta }_{\Lambda
\curvearrowright \Lambda ,D\otimes A}^{\widetilde{\tau }}$ with the von
Neumann-algebraic Bernoulli action $\beta _{\Lambda \curvearrowright \Lambda
,R}\colon \Lambda \rightarrow \mathrm{Aut}(R^{\otimes \Lambda })$, and $%
\overline{\beta }_{\Delta \curvearrowright \Lambda ,D\otimes A}^{\widetilde{%
\tau }}$ with $\beta _{\Delta \curvearrowright \Lambda ,R}\colon \Delta
\rightarrow \mathrm{Aut}(R^{\otimes \Lambda })$. Similarly, the extension of 
$(\delta \otimes \mathrm{id}_{A})^{\otimes \Lambda }$ to the weak closure
with respect to $\widetilde{\tau }$ can be identified with $(\overline{%
\delta }^{\tau _{D}}\otimes \mathrm{id}_{R})^{\overline{\otimes }\Lambda }$,
where $\tau _{D}$ is the unique trace on $D$. Furthermore, since $G$ is
compact, $\overline{B}^{\widetilde{\tau }}$ can be identified with the fixed
point algebra of $(\overline{\delta }^{\tau _{D}}\otimes \mathrm{id}_{R})^{%
\overline{\otimes }\Lambda }$, and the weak extension of $\widetilde{\alpha }
$ can be identified with the restriction of $\overline{\beta }^{\widetilde{%
\tau }}$ to $\overline{B}^{\widetilde{\tau }}$.

In the next proposition, we first show that $\alpha $ is conjugate to $%
\widetilde{\alpha }$. Then we use this alternative descriptions to verify
some properties of $\alpha $.

\begin{proposition}
\label{prop:AlternativeDescrAlpha} Adopt the notation of the discussion
above. Let $\tau _{0}$ be a trace on $A$, and $\tau $ be the image of $\tau
_{0}$ under the map $\iota $ from Proposition~\ref{prop:InvariantTracesAllG}%
. Define $\widetilde{\tau }$ to be the trace $(\tau _{D}\otimes \tau
_{0})^{\otimes \Lambda }$ on $(D\otimes A)^{\otimes \Lambda }$. Then:

\begin{enumerate}
\item There is a $\Lambda $-equivariant trace-preserving isomorphism $%
(A,\tau ,\alpha )\cong (B,\widetilde{\tau },\widetilde{\alpha })$.


\item There is a $\Lambda$-equivariant isomorphism $(A,\alpha)\cong
(A\otimes M_{p}^{\otimes \Lambda }, \alpha \otimes \beta _{\Lambda
\curvearrowright \Lambda ,M_{p}})$.

\item There is a $\Lambda $-equivariant trace-preserving isomorphism $%
(A,\tau ,\alpha )\cong (A\otimes A,\tau \otimes \tau ,\alpha \otimes \alpha
) $.

\item The action $\alpha $ is strongly outer;

\item The action $\overline{\alpha }^{\tau }$ is mixing.
\end{enumerate}
\end{proposition}

\begin{proof}
(1): By rearranging the tensor factors, it is clear that there exists a $%
\Lambda $-equivariant trace-preserving isomorphism 
\begin{equation*}
((D\otimes A)^{\otimes \Lambda },\widetilde{\tau },\beta _{\Lambda
\curvearrowright \Lambda ,D\otimes A})\cong (D^{\otimes \Lambda }\otimes
A^{\otimes \Lambda },\tau _{D}^{\otimes \Lambda }\otimes \tau _{0}^{\otimes
\Lambda },\beta _{\Lambda \curvearrowright \Lambda ,D}\otimes \beta
_{\Lambda \curvearrowright \Lambda ,A})\text{.}
\end{equation*}%
Upon identifying $A^{\otimes \Lambda }$ with $A$ via the isomorphism $\phi $%
, we obtain a $\Lambda $-equivariant isomorphism 
\begin{equation*}
((D\otimes A)^{\otimes \Lambda },\widetilde{\tau },\beta _{\Lambda
\curvearrowright \Lambda ,D\otimes A})\cong (D^{\otimes \Lambda }\otimes
A,\tau _{D}^{\otimes \Lambda }\otimes \tau ,\beta _{\Lambda \curvearrowright
\Lambda ,D}\otimes \sigma )\text{.}
\end{equation*}%
This isomorphism can be regarded as a $G$-equivariant isomorphism 
\begin{equation*}
((D\otimes A)^{\otimes \Lambda },\widetilde{\tau },(\delta \otimes \mathrm{id%
}_{A})^{\otimes \Lambda })\cong (D^{\otimes \Lambda }\otimes A,\tau
_{D}^{\otimes \Lambda }\otimes \tau ,\delta ^{\otimes \Lambda }\otimes 
\mathrm{id}_{A})\text{.}
\end{equation*}%
Upon taking $G$-fixed point algebras, and recalling that $E$ denotes the
fixed point algebra of $\delta ^{\otimes \Lambda }$, we obtain a
trace-preserving isomorphism $\psi \colon (B,\widetilde{\tau })\rightarrow
(E\otimes A,\tau _{E}\otimes \tau )$. Moreover, $\psi $ can be regarded as a 
$\Lambda $-equivariant trace-preserving isomorphism 
\begin{equation*}
\psi \colon (B,\widetilde{\tau },\widetilde{\alpha })\rightarrow (E\otimes
A,\tau _{E}\otimes \tau ,\beta _{\Lambda \curvearrowright \Lambda
,D}|_{E}\otimes \sigma ).
\end{equation*}%
Since $\xi \colon (A,\alpha )\rightarrow (E\otimes A,\beta _{\Lambda
\curvearrowright \Lambda ,D}|_{E}\otimes \sigma )$ is an equivariant
isomorphism by definition, and $(\tau _{E}\otimes \tau )\circ \xi =\tau $ by
Proposition~\ref{prop:InvariantTracesAllG}, it follows that $\xi ^{-1}\circ
\psi $ is a $\Lambda $-equivariant trace-preserving isomorphism $(B,%
\widetilde{\tau },\widetilde{\alpha })\rightarrow (A,\tau ,\alpha )$.

(2): By (1), it is enough to prove that there is a $\Lambda $-equivariant
isomorphism 
\begin{equation*}
(B,\widetilde{\alpha })\cong (B\otimes M_{p}^{\otimes \Lambda },\widetilde{%
\alpha }\otimes \beta _{\Lambda \curvearrowright \Lambda ,M_{p}})\text{.}
\end{equation*}
Using that $A$ is isomorphic to $A\otimes M_{p}$, it is clear that there is
an equivariant isomorphism 
\begin{equation*}
\left( (D\otimes A)^{\otimes \Lambda },\beta _{\Lambda \curvearrowright
\Lambda ,D\otimes A}\right) \cong \left( (D\otimes A)^{\otimes \Lambda
}\otimes M_{p}^{\otimes \Lambda },\beta _{\Lambda \curvearrowright \Lambda
,D\otimes A}\otimes \beta _{\Lambda \curvearrowright \Lambda ,M_{p}}\right) .
\end{equation*}%
This isomorphism intertwines the actions 
\begin{equation*}
(\delta \otimes \mathrm{id}_{A})^{\otimes \Lambda }\colon G\rightarrow 
\mathrm{Aut}\left( (D\otimes A)^{\otimes \Lambda }\right) \ \ \mbox{ and }\
\ (\delta \otimes \mathrm{id}_{A})^{\otimes \Lambda }\otimes \mathrm{id}%
_{M_{p}^{\otimes \Lambda }}\colon G\rightarrow \mathrm{Aut}\left( (D\otimes
A)^{\otimes \Lambda }\otimes M_{p}^{\otimes \Lambda }\right) .
\end{equation*}%
The fixed point algebra of the second action is isomorphic to $B\otimes
M_{p}^{\otimes \Lambda }$, in such a way that the restriction of $\beta
_{\Lambda \curvearrowright \Lambda ,D\otimes A}\otimes \beta _{\Lambda
\curvearrowright \Lambda ,M_{p}}$ to this algebra is conjugate to $%
\widetilde{\alpha }\otimes \beta _{\Lambda \curvearrowright \Lambda ,M_{p}}$%
. Thus $(B,\widetilde{\alpha })$ is equivariantly isomorphic to $(B\otimes
M_{p}^{\otimes \Lambda },\widetilde{\alpha }\otimes \beta _{\Lambda
\curvearrowright \Lambda ,M_{p}})$, as desired.

(3): This is similar to (2), using the fact that $(A^{\otimes
\Lambda},\tau_0^{\otimes\Lambda})\cong (A,\tau)$ via $\phi$.

(4): By (1) and (2), it suffices to show that $\widetilde{\alpha }\otimes
\beta _{\Lambda \curvearrowright \Lambda ,M_{p}}$ is strongly outer. Let $%
\gamma \in \Lambda \setminus \{1\}$ and let $\hat{\tau}$ be an $(\widetilde{%
\alpha }\otimes \beta _{\Lambda \curvearrowright \Lambda ,M_{p}})_{\gamma }$%
-invariant trace on $B\otimes M_{p}^{\otimes \Lambda }$. Since $%
M_{p}^{\otimes \Lambda }\cong D$ has a unique trace $\tau _{D}$, we deduce
that $\hat{\tau}$ has the form $\tau _{B}\otimes \tau _{D}$ for some $\alpha
_{\gamma }$-invariant trace $\tau _{B}$ on $B$. The weak extension of $(%
\widetilde{\alpha }\otimes \beta _{\Lambda \curvearrowright \Lambda
,M_{p}})_{\gamma }$ with respect to $\hat{\tau}$ is conjugate to $(\overline{%
\widetilde{\alpha }}^{\tau _{B}}\otimes \overline{\beta }_{\Lambda
\curvearrowright \Lambda ,M_{p}})_{\gamma }$, where $\overline{\beta }%
_{\Lambda \curvearrowright \Lambda ,M_{p}}$ is the von Neumann-algebraic
Bernoulli $(\Lambda \curvearrowright \Lambda )$-action with base $M_{p}$.
Such an action is easily seen to be outer.

(5): By (1), it suffices to check that $\overline{\widetilde{\alpha }}^{%
\widetilde{\tau }}$ is mixing. Observe that $\overline{\widetilde{\alpha }}^{%
\widetilde{\tau }}$ is the restriction to $\overline{B}^{\widetilde{\tau }}$
of $\overline{\beta }_{\Lambda \curvearrowright \Lambda ,D\otimes A}^{%
\widetilde{\tau }}$. The latter action is conjugate to the von
Neumann-algebraic Bernoulli action $\beta _{\Lambda \curvearrowright \Lambda
,\overline{D\otimes A}^{\tau _{0}}}$, which is mixing by Proposition~\ref%
{Proposition:bernoulli-wm}. Therefore $\overline{\widetilde{\alpha }}^{%
\widetilde{\tau }}$ is mixing.
\end{proof}

We retain the notation from before Proposition~\ref%
{prop:AlternativeDescrAlpha}. Given a trace $\tau $ in the image of the map $%
\iota $ from Proposition~\ref{prop:InvariantTracesAllG}, we will use the
pairing function $m_{\Delta }^{\overline{\alpha }^{\tau }}\colon H_{\Delta
,w}^{1}(\overline{\alpha }^{\tau })\times H_{\Delta ,w}^{1}(\overline{\alpha 
}^{\tau })\rightarrow H_{\Delta ,w}^{1}(\overline{\alpha }^{\tau }\otimes 
\overline{\alpha }^{\tau })$ from Definition~\ref{Definition:mTheta-relative}%
. As above, we write $\Gamma $ for the Pontryagin dual of $G$, and we let $%
m^{\Gamma }\colon \Gamma \times \Gamma \rightarrow \Gamma $ be the
multiplication operation. Recall also that $\mathtt{Lt}\colon G\rightarrow 
\mathrm{Aut}(C(G))$ denotes the action by left translation.

In what follows, and since all weak extensions will be taken with respect to 
$\widetilde{\tau }$, we will omit this trace from the notation. Also, we
will regard $\widetilde{\alpha }$ as an alternative description of $\alpha $%
, and, with a slight abuse of notation, we will write $\alpha $ to mean $%
\widetilde{\alpha }$. In particular, the symbol $\overline{\alpha }$ will
always represent the weak extension of $\widetilde{\alpha }$ with respect to
the trace $\widetilde{\tau }$.

\begin{theorem}
\label{thm:ComputCohom} Adopt the notation from Definition~\ref{df:alpha}.
Let $\tau _{0}\in T(A)$, and set $\tau =\iota (\tau _{0})$. Then there exist
bijections $\eta \colon H_{\Delta ,w}^{1}(\overline{\alpha }^{\tau
})\rightarrow \Gamma $ and $\eta ^{(2)}\colon H_{\Delta ,w}^{1}(\overline{%
\alpha }^{\tau }\otimes \overline{\alpha }^{\tau })\rightarrow \Gamma $
satisfying 
\begin{equation*}
\eta ^{(2)}\circ m_{\Delta }^{\overline{\alpha }^{\tau }}=m^{\Gamma }\circ
(\eta \times \eta ).
\end{equation*}
\end{theorem}

\begin{proof}
Let $\tau _{D}$ denote the unique trace on $D$. For the trace $\tau _{0}$ in
the statement, set $\widetilde{\tau }=(\tau _{D}\otimes \tau _{0})^{\otimes
\Lambda }$, which is naturally a trace on $(D\otimes A)^{\otimes \Lambda }$.
Recall the definition of $\widetilde{\alpha }$ from before Proposition~\ref%
{prop:AlternativeDescrAlpha}. By part (1) of Proposition~\ref%
{prop:AlternativeDescrAlpha} the action $\alpha $ is conjugate to $%
\widetilde{\alpha }$ via an isomorphism that maps $\tau $ to $\widetilde{%
\tau }$. In particular, the weak extension of $\alpha $ with respect to $%
\tau $ is conjugate to the weak extension of $\widetilde{\alpha }$ with
respect to $\widetilde{\tau }$. Therefore it is enough to prove the
corresponding statement for $\widetilde{\alpha }$ and $\widetilde{\tau }$.

Let $\beta $ be the Bernoulli $\left( \Lambda \curvearrowright \Lambda
\right) $-action with base $D\otimes A$, and $\zeta $ its restriction to $%
\Delta $ which is the Bernoulli $\left( \Delta \curvearrowright \Lambda
\right) $-action with base $D\otimes A$. Observe that $\widetilde{\alpha }$
is the restriction of $\beta $ to $B\subseteq (D\otimes A)^{\otimes \Lambda
} $, and the unique extension $\overline{\widetilde{\alpha }}$ of $%
\widetilde{\alpha }$ to the weak closure $\overline{B}$ with respect to the
trace $\widetilde{\tau }$. Therefore the conclusion follows from Lemma \ref%
{Lemma:key}.
\end{proof}

Using the previous result, we will show below that if one starts with two
non-isomorphic abelian pro-$p$ groups $G_{0}$ and $G_{1}$, then the actions $%
\alpha^{G_{0}}$ and $\alpha ^{G_{1}}$ of $\Lambda $ on $A$, as in Definition %
\ref{df:alpha}, are not weakly cocycle conjugate. Since the pro-$p$ group is
no longer fixed, we again use superscripts (for the actions) and subscripts
(for the algebras) to keep track of which pro-$p$ group they come from.

\begin{theorem}
\label{Theorem:conjugacy} Let the notation be as in Definition~\ref{df:alpha}
and Proposition~\ref{prop:InvariantTracesAllG}. Fix $\tau_0\in T(A)$ and set 
$\tau=\iota(\tau)$. Let $G_{0}$ and $G_{1}$ be second countable abelian pro-$%
p$ groups. The following assertions are equivalent:

\begin{enumerate}
\item The groups $G_{0}$ and $G_{1}$ are topologically isomorphic;

\item The $\Lambda$-actions $\alpha ^{G_{0}}$ and $\alpha ^{G_{1}}$ are
conjugate;

\item The $\Lambda$-actions $\alpha ^{G_{0}}$ and $\alpha ^{G_{1}}$ are
cocycle conjugate;




\item The $\Lambda $-actions $\alpha ^{G_{0}}$ and $\alpha ^{G_{1}}$ are
weakly cocycle $\tau $-conjugate.
\end{enumerate}
\end{theorem}

\begin{proof}
Since $G_{0}$ and $G_{1}$ are pro-$p$ groups, there are isomorphisms $%
D_{G_{0}}\cong D_{G_{1}}\cong M_{p^{\infty }}$, and we denote this algebra
by $D$. Any isomorphism $G_{0}\cong G_{1}$ is immediately seen to induce an
equivariant *-isomorphism $(D,\delta ^{G_{0}})\cong (D,\delta ^{G_{1}})$.
From this, it is easy to construct an equivariant *-isomorphism $%
(B_{G_{0}},\alpha ^{G_{0}})\cong (B_{G_{1}},\alpha ^{G_{1}})$. This proves
the implication (1)$\Rightarrow $(2). It is clear that (2) implies (3), and
that (3) implies (4), because $\tau $ is $\alpha ^{G_{0}}$- and $\alpha
^{G_{1}}$-invariant by Proposition~\ref{prop:InvariantTracesAllG}. 
Therefore, it only remains to prove that (4) implies (1).

Suppose that $\alpha ^{G_{0}}$ and $\alpha ^{G_{1}}$ are weakly cocycle $%
\tau $-conjugate. Let $\tau _{D}$ denote the unique trace on $D$, and set $%
\widetilde{\tau }=(\tau _{D}\otimes \tau _{0})^{\otimes \Lambda }$, which is
a trace on $(D\otimes A)^{\otimes \Lambda }$. By part~(1) of Proposition~\ref%
{prop:AlternativeDescrAlpha}, for $i=0,1$, the weak extension of $\alpha
^{G_{i}}$ with respect to $\tau $ can be identified with the weak extension
of $\widetilde{\alpha }^{G_{i}}$ with respect to $\widetilde{\tau }$. (The
action $\widetilde{\alpha }^{G_{i}}$ is described before Proposition~\ref%
{prop:AlternativeDescrAlpha}.) Hence, it suffices to show that $\widetilde{%
\alpha }^{G_{0}}$ and $\widetilde{\alpha }^{G_{1}}$ are not weakly cocycle $%
\widetilde{\tau }$-conjugate. Since the trace $\widetilde{\tau }$ is fixed,
we will omit it from the notation for weak closures and weak extensions.
With a slight abuse of notation, we will write $\alpha ^{G_{i}}$ to mean $%
\widetilde{\alpha }^{G_{i}}$. This way, the symbol $\overline{\alpha }%
^{G_{i}}$ will always represent the weak extension of $\widetilde{\alpha }%
^{G_{i}}$ with respect to the trace $\widetilde{\tau }$. Let $\beta $ be the
Bernoulli $\left( \Lambda \curvearrowright \Lambda \right) $-action with
base $D\otimes A$. We also denote the algebra $\overline{(D\otimes A)}%
^{\otimes \Lambda }$ by $N$ (omitting the trace $\tau _{D}\otimes \tau _{0}$%
), and abbreviate the $G_{i}$-action $(\overline{\delta }^{G_{i}}\otimes 
\mathrm{id}_{R})^{\overline{\otimes }\Lambda }$ on $N$ to $\rho ^{(i)}\colon
G_{i}\rightarrow \mathrm{Aut}(N)$. (In particular, $\overline{B}%
_{G_{i}}=N^{\rho _{i}}$.) We let $\Gamma _{i}$ be the dual group of $G_{i}$.

Let $\psi \colon \overline{B}_{G_{0}}\rightarrow \overline{B}_{G_{1}}$ be an
isomorphism and let $w\colon \Lambda \rightarrow U(\overline{B}_{G_{1}})$ be
a weak $1$-cocycle for $\alpha ^{G_{1}}$ satisfying 
\begin{equation}
\mathrm{Ad}(w_{\gamma })\circ \overline{\alpha }_{\gamma }^{G_{1}}=\psi
\circ \overline{\alpha }_{\gamma }^{G_{0}}\circ \psi ^{-1}\text{\label{eqn}}
\end{equation}%
for every $\gamma \in \Lambda $. Using Popa's superrigidity theorem \cite[%
Theorem 4.1]{popa_some_2006} in the case of weak $1$-cocycles as in the
proof of Theorem~\ref{thm:ComputCohom}, one can find unitaries $z\in U(%
\overline{B}_{G_{1}})$ and $v\in U(N)$, and a character $\chi \in \Gamma
_{1} $ such that 
\begin{equation*}
w_{\gamma }=z^{\ast }v^{\ast }\overline{\beta }_{\gamma }(v)\overline{\alpha 
}_{\gamma }(z)\ \mathrm{\mathrm{mod}}\mathbb{C}\ \ \mbox{
and }\ \ \rho _{g}^{(1)}(v)=\chi (g)v
\end{equation*}%
for every $\gamma \in \Delta $ and for every $g\in G_{1}$. Therefore, upon
replacing $\psi $ with $\psi \circ \mathrm{\mathrm{Ad}}(z^{\ast })$, we can
assume that $z=1$ and $w_{\gamma }=v^{\ast }\overline{\beta }_{\gamma }(v)$
for every $\gamma \in \Delta $.

Next, we want to define a bijection $\varphi \colon H_{\Delta ,w}^{1}(%
\overline{\alpha }^{G_{0}})\rightarrow H_{\Delta ,w}^{1}(\overline{\alpha }%
^{G_{1}})$. Given a function $u\colon \Lambda \rightarrow U(\overline{B}%
_{G_{1}})$, define $\psi (u)w\colon \Lambda \rightarrow U(\overline{B}%
_{G_{1}})$ to be the function given by $(\psi (u)w)_{\gamma }=\psi
(u_{\gamma })w_{\gamma }$ for all $\gamma \in \Lambda $.

\begin{claim*}
If $u\in Z_{w}^{1}(\overline{\alpha }^{G_{0}})$, then $\psi (u)w\in
Z_{w}^{1}(\overline{\alpha }^{G_{1}})$.
\end{claim*}

\begin{proof}[Proof of claim]
Let $\gamma ,\sigma \in \Lambda $. In the following computation (where all
equalities are up to scalars), we use the fact that $w$ is a weak $1$%
-cocycle for $\overline{\alpha }^{G_{1}}$ at the first step, and equation~(%
\ref{eqn}) at the second step, to get 
\begin{equation*}
\psi (u_{\gamma \sigma })=\psi (u_{\gamma })(\psi \circ \overline{\alpha }%
_{\gamma }^{G_{0}}\circ \psi ^{-1})(\psi (u_{\sigma }))=\psi (u_{\gamma })(%
\mathrm{\mathrm{Ad}}(w_{\gamma })\circ \overline{\alpha }_{\gamma
}^{G_{1}})(\psi (u_{\sigma }))=\psi (u_{\gamma })w_{\gamma }\overline{\alpha 
}_{\gamma }^{G_{1}}(\psi (u_{\sigma }))w_{\gamma }^{\ast }\ \ \mathrm{mod}%
\mathbb{C}
\end{equation*}%
Therefore, using the above identity and the fact that $u$ is a weak $1$%
-cocycle for $\overline{\alpha }^{G_{0}}$, we deduce that 
\begin{equation*}
\psi (u_{\gamma \sigma })w_{\gamma \sigma }=\psi (u_{\gamma })w_{\gamma }%
\overline{\alpha }_{\gamma }^{G_{1}}(\psi (u_{\sigma }))w_{\gamma }^{\ast
}w_{\gamma }\overline{\alpha }_{\gamma }^{G_{1}}(w_{\sigma })=\psi
(u_{\gamma })w_{\gamma }\overline{\alpha }_{\gamma }^{G_{1}}(\psi (u_{\sigma
})w_{\sigma })\ \ \mathrm{mod}\mathbb{C}\text{.}
\end{equation*}%
This shows that $\psi (u)w$ is a weak $1$-cocycle for $\overline{\alpha }%
^{G_{1}}$, proving the claim.
\end{proof}

It follows that there is a well-defined map $\widehat{\psi }\colon Z_{w}^{1}(%
\overline{\alpha }^{G_{0}})\rightarrow Z_{w}^{1}(\overline{\alpha }^{G_{1}})$
given by $\widehat{\psi }(u)=\psi (u)w$ for $u\in Z_{w}^{1}(\overline{\alpha 
}^{G_{0}})$.

\begin{claim*}
If $u,u^{\prime }\in Z_{w}^{1}(\overline{\alpha }^{G_{0}})$ are $\Delta $%
-locally weakly cohomologous, then so are $\widehat{\psi }(u)$ and $\widehat{%
\psi }(u^{\prime })$.
\end{claim*}

\begin{proof}[Proof of claim]
Find a unitary $z\in U(\overline{B}_{G_{0}})$ satisfying $u_{\gamma
}^{\prime }=z^{\ast }u_{\gamma }\overline{\alpha }_{\gamma }^{G_{0}}(z)\ 
\mathrm{mod}\mathbb{C}$ for $\gamma \in \Delta $. Then 
\begin{equation*}
\psi (u_{\gamma }^{\prime })w_{\gamma }=\psi (z)^{\ast }\psi (u_{\gamma
})(\psi \circ \overline{\alpha }_{\gamma }^{G_{0}}\circ \psi
^{-1})(z)w_{\gamma }=\psi (z)^{\ast }\psi (u_{\gamma })(\mathrm{Ad}%
(w_{\gamma })\circ \overline{\alpha }_{\gamma }^{G_{1}})(z)w_{\gamma }=\psi
(z)^{\ast }\psi (u_{\gamma })w_{\gamma }\overline{\alpha }_{\gamma
}^{G_{1}}(z)\ \mathrm{mod}\mathbb{C}
\end{equation*}%
for all $\gamma \in \Delta $. This shows that $\psi (u^{\prime })w$ and $%
\psi (u)w$ are $\Delta $-locally weakly cohomologous.
\end{proof}

It follows that $\widehat{\psi }$ induces a well-defined map $\varphi \colon
H_{\Delta ,w}^{1}(\overline{\alpha }^{G_{0}})\rightarrow H_{\Delta ,w}^{1}(%
\overline{\alpha }^{G_{1}})$.

\begin{claim*}
The map $\varphi $ is invertible.
\end{claim*}

\begin{proof}[Proof of claim]
It follows from equation~(\ref{eqn}) that 
\begin{equation*}
\overline{\alpha }_{\gamma }^{G_{0}}=\psi ^{-1}\circ \mathrm{Ad}(w_{\gamma
})\circ \overline{\alpha }_{\gamma }^{G_{1}}\circ \psi =\mathrm{Ad}(\psi
^{-1}(w_{\gamma }))\circ \psi ^{-1}\circ \overline{\alpha }_{\gamma
}^{G_{0}}\circ \psi
\end{equation*}%
for all $\gamma \in \Lambda $. Therefore, the same argument as before shows
that the function that assigns to the cocycle $u$ for $\overline{\alpha }%
^{G_{1}}$ the cocycle $\gamma \mapsto \psi ^{-1}\left( u_{\gamma }w_{\gamma
}\right) $ for $\overline{\alpha }^{G_{0}}$ induces a well-defined function $%
H_{\Delta ,w}^{1}(\overline{\alpha }^{G_{1}})\rightarrow H_{\Delta ,w}^{1}(%
\overline{\alpha }^{G_{0}})$, which is easily seen to be the inverse of $%
\varphi $. This proves the claim.
\end{proof}

Similarly as above, we define a bijection $\varphi ^{(2)}\colon H_{\Delta
,w}^{1}(\overline{\alpha }^{G_{0}}\otimes \overline{\alpha }%
^{G_{0}})\rightarrow H_{\Delta ,w}^{1}(\overline{\alpha }^{G_{1}}\otimes 
\overline{\alpha }^{G_{1}})$, by 
\begin{equation*}
\varphi ^{(2)}([u])=[(\psi \otimes \psi )(u)(w\otimes w)]
\end{equation*}%
for all $u\in Z_{w}^{1}(\overline{\alpha }\otimes \overline{\alpha })$,
where $(\psi \otimes \psi )(u)(w\otimes w)\colon \Lambda \rightarrow U(%
\overline{B}_{G_{1}})$ is the weak $1$-cocycle for $\overline{\alpha }%
^{G_{1}}\otimes \overline{\alpha }^{G_{1}}$ given by $\gamma \mapsto (\psi
\otimes \psi )(u_{\gamma })(w_{\gamma }\otimes w_{\gamma })$ for all $\gamma
\in \Lambda $. Moreover, a routine calculation shows that $\varphi
^{(2)}\circ m_{\Delta }^{\overline{\alpha }^{G_{0}}}=m_{\Delta }^{\overline{%
\alpha }^{G_{1}}}\circ \varphi $. For $i\in \left\{ 0,1\right\} $, let $\eta
_{G_{i}}\colon H_{\Delta ,w}^{1}(\overline{\alpha }^{G_{i}})\rightarrow
\Gamma _{i}$ and $\eta _{G_{i}}^{(2)}\colon H_{\Delta ,w}^{1}(\overline{%
\alpha }^{G_{i}}\otimes \overline{\alpha }^{G_{i}})\rightarrow \Gamma _{i}$
be the maps from Theorem~\ref{thm:ComputCohom}, and set 
\begin{equation*}
\pi =\eta _{G_{1}}\circ \varphi \circ \eta _{G_{0}}^{-1}\colon \Gamma
_{0}\rightarrow \Gamma _{1}\ \ \mbox{ and }\ \ \pi ^{(2)}=\eta
_{G_{1}}^{(2)}\circ \varphi ^{(2)}\circ (\eta _{G_{0}}^{(2)})^{-1}\colon
\Gamma _{0}\rightarrow \Gamma _{1}.
\end{equation*}%
By Theorem~\ref{thm:ComputCohom}, the following diagram is commutative: 
\begin{equation*}
\xymatrix{ \Gamma_0\times \Gamma_0\ar[d]_-{\pi\times\pi} &&H^1_{\Delta,
w}(\overline{\alpha}^{G_0})\times H^1_{\Delta, w}(\overline{\alpha}^{G_0})
\ar[ll]_-{\eta_{G_0}\times \eta_{G_0}}\ar[rr]^-{m^{\overline{\alpha
}^{G_{0}}}}\ar[d]_-{\varphi\times\varphi}&& H^1_{\Delta,
w}(\overline{\alpha}^{G_0}\otimes
\overline{\alpha}^{G_0})\ar[d]^-{\varphi^{(2)}}\ar[rr]^-{\eta^{(2)}_{G_0}}
&& \Gamma_0\ar[d]^-{\pi ^{(2)}}\\ \Gamma_1 \times \Gamma_1 &&H^1_{\Delta,
w}(\overline{\alpha}^{G_1})\times H^1_{\Delta, w}(\overline{\alpha}^{G_1})
\ar[ll]^-{\eta_{G_1}\times \eta_{G_1}}\ar[rr]_--{m^{\overline{\alpha
}^{G_{1}}}}&& H^1_{\Delta, w}(\overline{\alpha}^{G_1}\otimes
\overline{\alpha}^{G_1}) \ar[rr]^-{\eta^{(2)}_{G_1}} && \Gamma_1. }
\end{equation*}%
Recall that $\chi $ denotes the character of $G_{1}$ associated with the
weak $1$-cocycle $w$ for $\overline{\alpha }^{G_{1}}$. Then $\pi
^{(2)}(\omega \omega ^{\prime })=\pi (\omega )\pi (\omega ^{\prime })$ and $%
\pi (1_{\Gamma _{0}})=\chi $. It follows that $\pi ^{(2)}(\omega )=\pi
(\omega )\chi $ for every $\omega \in \Gamma _{0}$. Therefore the map $%
\widetilde{\pi }\colon \Gamma _{0}\rightarrow \Gamma _{1}$ given by $%
\widetilde{\pi }(\omega )=\pi (\omega )\chi ^{-1}$ for all $\omega \in
\Gamma _{0}$, is a group isomorphism. Indeed, we have%
\begin{equation*}
\pi (\omega )\chi ^{-1}\pi (\omega ^{\prime })\chi ^{-1}=\pi ^{(2)}(\omega
\omega ^{\prime })\chi ^{-2}=\pi (\omega \omega ^{\prime })\chi ^{-1}
\end{equation*}%
for $\omega ,\omega ^{\prime }\in \Gamma _{0}$. Since clearly $\widetilde{%
\pi }$ is a bijection, we conclude that $\widetilde{\pi }$ is a group
isomorphism, and hence $\Gamma _{0}\cong \Gamma _{1}$. By Pontryagin
duality, we conclude that $G_{0}\cong G_{1}$, and the proof is finished.
\end{proof}

We now arrive at the main result of this section. Its conclusion will be
significantly strengthened in Corollary~\ref{Corollary:conjugacy}.

\begin{theorem}
\label{Theorem:uncountably} Let $\Lambda $ be a countable discrete group
with an infinite relative property \emph{(T) }subgroup, let $p$ be a prime
number, and let $A$ be separable, locally reflexive, $M_{p^{\infty }}$%
-absorbing, unital C*-algebra admitting an amenable trace, and such that $%
A\cong A^{\otimes \mathbb{N}}$. Then there exists a continuum $(\alpha
^{(t)})_{t\in \mathbb{R}}$ of pairwise not weakly cocycle conjugate,
strongly outer actions of $\Lambda $ on $A$. In fact, there exists an
amenable, extreme trace $\tau $ that is invariant under $\alpha ^{(t)}$ for
every $t\in \mathbb{R}$, and such that the actions $\alpha ^{(t)}$ are all $%
\tau $-mixing and pairwise not weakly cocycle $\tau $-conjugate.
\end{theorem}

\begin{proof}
Let $(G_{t})_{t\in \mathbb{R}}$ be a continuum family of pairwise
nonisomorphic abelian pro-$p$ groups. For $t\in \mathbb{R}$, set $\alpha
^{(t)}:=\alpha ^{G_{t}}$, where $\alpha ^{G_{t}}:\Lambda \rightarrow \mathrm{%
Aut}\left( A\right) $ is the action of $\Lambda $ on $A$ given by Definition %
\ref{df:alpha}. By part~(4) of Proposition~\ref{prop:AlternativeDescrAlpha}, 
$\alpha ^{(t)}$ is strongly outer. Since $A$ has an amenable trace and $T_{%
\mathrm{am}}(A)$ is a face in the simplex $T(A)$, there exists an extreme,
amenable trace $\tau _{0}$ on $A$. Let $\iota \colon T(A)\rightarrow T(A)$
be the map from Proposition~\ref{prop:InvariantTracesAllG}. Then $\tau
=\iota (\tau _{0})$ is extreme and amenable, and it is $\alpha ^{(t)}$%
-invariant for every $t\in \mathbb{R}$ by Proposition~\ref%
{prop:InvariantTracesAllG}. By part~(5) of Proposition~\ref%
{prop:AlternativeDescrAlpha}, $\alpha ^{(t)}$ is $\tau $-mixing for every $%
t\in \mathbb{R}$. Finally, Theorem~\ref{Theorem:conjugacy} implies that the
weak extensions of the $\alpha ^{(t)}$ to $\overline{A}^{\tau }$ are
pairwise not weakly cocycle conjugate. This concludes the proof.
\end{proof}

We make some comments on the assumptions of the theorem above. First,
subgroups with relative property (T) are abundant: if either $\Lambda$ or $%
\Delta$ has property (T), then the inclusion $\Delta\subseteq \Lambda$ has
relative property (T). On the other hand, it is easy to find many
C*-algebras satisfying the assumptions of Theorem~\ref{Theorem:uncountably}.
Indeed, if $A_{0}$ is any separable, unital, exact C*-algebra with an
amenable trace, then $A=M_{p}^{\infty }\otimes A_{0}^{\otimes \mathbb{N}}$
satisfies the assumptions of said theorem. In particular, $A_{0}$ and $A$
need not be simple. We also remark that every trace on a nuclear C*-algebra
is necessarily amenable.


To end this section, we explicitly state our result for UHF-algebras, to
highlight the contrast with the results in \cite%
{kishimoto_rohlin_1995,matui_actions_2011,szabo_strongly_2017}.


\begin{corollary}
\label{cor:UHFuncountable} Let $D$ be a UHF-algebra of infinite type, and
let $\Lambda$ be a countable group with an infinite subgroup with relative
property (T). Then there exists a continuum of pairwise non (weakly)
cocycle-conjugate, strongly outer actions of $\Lambda$ on $D$.
\end{corollary}

\section{Conjugacy, cocycle conjugacy, and weak cocycle conjugacy are not
Borel \label{Section:conjugacy}}

In this section, we discuss how the construction from Section~4 can be used
to prove that, under the assumptions of Theorem \ref{Theorem:uncountably},
conjugacy, cocycle conjugacy, and weak cocycle conjugacy of strongly outer
actions of $\Lambda $ on $A$ are complete analytic sets.

\subsection{Borel complexity of equivalence relations\label%
{Sbs:borel-complexity}}

We recall here some notions from Borel complexity theory. In this setting, a 
\emph{classification problem }is identified with an equivalence relation $E$
on a Polish space $X$. Virtually any concrete classification problem in
mathematics is of this form, perhaps after a suitable parameterization. For
example, a countable discrete group can be identified with a set of triples
of natural numbers, coding a group operation on $\mathbb{N}$. The space of
such sets of triples is a $G_{\delta }$ subset of the compact metrizable
space $\left\{ 0,1\right\} ^{\mathbb{N}^{3}}$ endowed with the product
topology. (A $G_{\delta }$ subspace of a Polish space is Polish by \cite[%
Theorem 3.11]{kechris_classical_1995}.)

\begin{definition}
(See \cite[Definitions 5.1.1 and 5.1.2]{gao_invariant_2009}). A \emph{Borel
reduction} from an equivalence relation $E$ on a Polish space $X$ to an
equivalence relation $F$ on a Polish space $Y$ is a Borel function $f\colon
X\rightarrow Y$ such that $[x] _{E}\mapsto [ f(x)] _{F}$ is a well-defined
injective function from the space $X/E$ of $E$-classes to the space $Y/F$ of 
$F$-classes. The equivalence relation $E$ is said to be \emph{Borel reducible%
} to $F$, in formulas $E\leq _{B}F$, if there exists a Borel reduction from $%
E$ to $F$.
\end{definition}

\begin{remark}
When $E$ is Borel reducible to $F$, the objects of $X$ up to $E$ can be 
\emph{explicitly }classified using $F$-classes as complete invariants. In
other words, the classification problem represented by $F$ is \emph{at least
as complex }as the classification problem represented by $E$. (Observe that
this notion does not depend on the topologies of $X$ and $Y$, but only on
the standard Borel structures that they induce.)
\end{remark}

The notion of Borel reducibility can be used to measure the complexity of a
given classification problem. The first natural measure of complexity is
simply the \emph{number of classes }of the corresponding equivalence
relation. Theorem \ref{Theorem:uncountably} addresses this problem in the
case of conjugacy, cocycle conjugacy, and weak cocycle conjugacy of strongly
outer actions of $\Lambda $ on $A$: they have a continuum of equivalence
classes.

The natural next step in the study of the complexity of a classification
problem consists in determining whether the classes can be \emph{explicitly }%
parameterized as the points of a Polish space. This is equivalent to the
corresponding equivalence relation being \emph{smooth}, that is, Borel
reducible to the relation of \emph{equality }in some Polish space. As an
example, Glimm's classification of separable UHF-algebras implies that the
relation of *-isomorphism for these algebras is smooth (even though there
exists a continuum of isomorphism classes). Similarly, the orbit equivalence
relation of a continuous action of a compact group on a Polish space is
smooth. However, isomorphism of countable rank-one torsion-free abelian
groups, for instance, is not smooth. Another canonical example of a
nonsmooth equivalence relation is the relation of \emph{tail equivalence }%
for binary sequences.

A more generous notion of being well-behaved for an equivalence relation $E$
on $X$ is being \emph{Borel }as a subset of $X\times X$. For instance,
isomorphism of countable rank-one torsion-free abelian groups is Borel.
Similarly, tail equivalence of binary sequences is Borel and, more
generally, the orbit equivalence relation of a \emph{free }continuous action
of a Polish group on a Polish space is Borel. (The orbit equivalence
relation of a continuous action of a Polish group $G$ on a Polish pace $X$
is Borel if and only if the map that assigns to each point $x$ of $X$ the
corresponding \emph{stabilizer subgroup }$G_{x}$ of $G$ is Borel; see \cite[%
7.1.2]{becker_descriptive_1996}.) Since the relation of equality on any
Polish space is clearly Borel, any smooth equivalence relation is, in
particular, Borel.

One can also define a similar notion of comparison among \emph{sets}, rather
than equivalence relations.

\begin{definition}
(See \cite[Section 14.A and Definition 26.7]{kechris_classical_1995}.) A
subset $A$ of a Polish space $X$ is said to be \emph{analytic}, or $%
\boldsymbol{\Sigma }_{1}^{1}$, if there exist a Polish space $Z$ and a Borel
function $f\colon Z\rightarrow X$ such that $A$ is the image under $f$ of a
Borel subset of $Z$. A \emph{complete analytic set }(also called $%
\boldsymbol{\Sigma }_{1}^{1}$-complete set)\emph{\ }is an analytic subset $A$
of a Polish space $X$ such that, for any other analytic subset $B$ of a
Polish space $Y$, there exists a Borel function $f\colon Y\rightarrow X$
such that $f^{-1}(A)=B$.
\end{definition}

We recall here the fundamental fact that a complete analytic set is not
Borel. The canonical example of a complete analytic set is the set of
ill-founded trees on $\mathbb{N}$; see \cite[Section 27.A]%
{kechris_classical_1995}.

As above, we regard an equivalence relation $E$ on a Polish space $X$ as a
subset of the product space $X\times X$ endowed with the product topology.
Consistently, we say that $E$ is a complete analytic set if it is complete
analytic as a subset of $X\times X$. It is clear that if $E$ is Borel
reducible to an equivalence relation $F$, and $E$ is a complete analytic
set, then $F$ is a complete analytic set as well.

In Theorem~\ref{Theorem:conjugacy2}, we will prove that the construction of
actions from profinite groups described in Section \ref{Sct:uncountably} can
be used to show that, under the assumption of Theorem \ref%
{Theorem:uncountably}, the relations of conjugacy, cocycle conjugacy, and
weak cocycle conjugacy of strongly outer actions of $\Lambda $ on $A$ are
complete analytic sets. This is a significant strengthening of the
conclusions of Theorem \ref{Theorem:uncountably}.

\subsection{Parametrizing actions\label{Sbs:actions}}

For the rest of this section, we fix a countable discrete group $\Lambda $
and a separable unital C*-algebra $A$. We proceed to explain how the
classification problem for strongly outer actions of $\Lambda $ on $A$ can
be naturally regarded as equivalence relations on a Polish space. We regard $%
\mathrm{T}(A)$ as a compact metrizable space endowed with the w*-topology.
The set $\mathrm{Act}_{\Lambda }(A)$ of actions of $\Lambda $ on $A$ is a
closed subset of the product space $\mathrm{Aut}(A)^{\Lambda }$ endowed with
the product topology, giving it the structure of a Polish space.

\begin{notation}
\label{nota:SpacesActions} Let $\tau $ be a trace on $A$.

\begin{itemize}
\item We denote by $\mathrm{Act}_\Lambda(A, \tau)$ the set of $\tau $%
-preserving actions of $\Lambda $ on $A$;

\item We denote by $\mathrm{WM}_\Lambda(A, \tau)$ the set of $\tau $%
-preserving weakly $\tau $-mixing actions of $\Lambda $ on $A$;

\item We denote by $\mathrm{SO}_{\Lambda }(A)$ the set of strongly outer
actions of $\Lambda $ on $A$;

\item We denote by $\mathrm{SOWM}_{\Lambda }(A,\tau )$ the set of $\tau $%
-preserving weakly $\tau $-mixing strongly outer actions of $\Lambda $ on $A$%
.
\end{itemize}
\end{notation}

It is easy to see that $\mathrm{Act}_{\Lambda }(A, \tau )$ and $\mathrm{WM}%
_{\Lambda }(A, \tau )$ are $G_{\delta }$ subsets of $\mathrm{Act}_{\Lambda
}(A)$. We will show below that $\mathrm{SO}_{\Lambda }(A)$ and $\mathrm{SOWM}%
_{\Lambda }(A, \tau )$ are also $G_{\delta }$ subsets of $\mathrm{Act}%
_{\Lambda }\left( A\right) $.

Given a C*-algebra $A$, we let $A_{\mathrm{sa}}$ be the set of selfadjoint
elements of $A$. An element $a$ of $A_{\mathrm{sa}}$ is a contraction if $%
\left\Vert a\right\Vert \leq 1$. Given a trace $\tau $ on $A$, we let $%
\left\Vert a\right\Vert _{\tau }=\sqrt{\tau \left( a^{\ast }a\right) }$ be
the $2$-norm induced by $\tau $ on $A$ and $\overline{A}^{\tau }$. Using
Borel functional calculus \cite[Section I.4.3]{blackadar_operator_2006}, we
fix a continuous function $\varpi \colon \lbrack 0,+\infty )\rightarrow
\lbrack 0,+\infty )$ satisfying the following properties:

\begin{itemize}
\item $\varpi(0)=0$ and $\varpi(t)\geq t$ for all $t\in [0, \infty)$;

\item Let $(M,\tau )$ be a tracial von Neumann algebra, let $a\in M_{\mathrm{%
sa}}$ be a contraction, and let $\varepsilon >0$. If $\Vert a^{2}-a\Vert
_{\tau }<\varepsilon $, then there exists a projection $p\in M$ such that $%
\Vert p-a\Vert _{\tau }<\varpi (\varepsilon )$.
\end{itemize}

The following lemma will be used to show that the space of strongly outer
(weak mixing) actions of a fixed countable group on a unital, separable
C*-algebra is a Polish space; see Proposition~\ref{Proposition:free-G}. Its
proof follows from Lemma~4.2 using an easy approximation argument, which is
presented for the sake of completeness.

\begin{lemma}
\label{Lemma:free} Let $A$ be a unital, separable C*-algebra, let $\theta
\in \mathrm{Aut}\left( A\right) $ be an automorphism of $A$, and let $\tau $
be a $\theta $-invariant trace on $A$. Fix a countable dense subset $A_{0}$
of the unit ball of $A_{\mathrm{sa}}$. Then $\overline{\theta }^{\tau }$ is
properly outer if and only if for every $a\in A_{0}$, and every $\varepsilon
\in (0,+\infty )\cap \mathbb{Q}$ satisfying $\Vert a^{2}-a\Vert _{\tau }\leq
\varepsilon $, there exists a contraction $b\in A_{\mathrm{sa}}$ satisfying 
\begin{equation*}
\Vert b^{2}-b\Vert _{\tau }<\varepsilon ,\ \ \Vert ab-b\Vert _{\tau }<\varpi
(\varepsilon ),\ \ \Vert b\theta (b)\Vert _{\tau }<\varepsilon ,\ \ 
\mbox{
and }\ \ \tau (b)>\frac{1}{3}\tau (a)-\varepsilon .
\end{equation*}
\end{lemma}

\begin{proof}
Let $\pi _{\tau }:A\rightarrow B\left( L^{2}\left( A,\tau \right) \right) $
be the GNS representation associated with $\tau $. Suppose that $\overline{%
\theta }^{\tau }$ is properly outer. Let $\varepsilon \in (0,\infty )\cap 
\mathbb{Q}$, and let $a\in A_{0}$ satisfy $\Vert a^{2}-a\Vert _{\tau
}<\varepsilon $. By the choice of $\varpi $, there exists a projection $p\in 
\overline{A}^{\tau }$ such that $\Vert \pi _{\tau }(a)-p\Vert _{\tau
}<\varpi (\varepsilon )$. By (1)$\Rightarrow $(4) in \cite[Lemma 4.2]%
{kerr_turbulence_2010}, there exists a projection $q\in \overline{A}^{\tau }$
such that 
\begin{equation*}
q\leq p,\ \ \Vert q\overline{\theta }^{\tau }(q)\Vert _{\tau }<\varepsilon
,\ \ \mbox{ and }\ \ \tau (q)\geq \frac{1}{3}\tau (p)>\frac{1}{3}\tau
(a)-\varepsilon .
\end{equation*}%
Therefore, $\left\Vert q\pi _{\tau }(a)-q\right\Vert _{\tau }\leq \left\Vert
q-qp\right\Vert _{\tau }+\left\Vert qp-q\pi _{\tau }(a)\right\Vert _{\tau
}<\varpi (\varepsilon )$, and similarly $\left\Vert \pi _{\tau
}(a)q-q\right\Vert _{\tau }<\varpi (\varepsilon )$. Since the norm-unit ball
of $A$ is $\Vert \cdot \Vert _{\tau }$-dense in the unit ball of $\overline{A%
}^{\tau }$, there exists a contraction $b\in A_{\mathrm{sa}}$ satisfying the
conditions in the statement.

We prove the converse. We want to prove that $\overline{\theta }^{\tau }$ is
properly outer. Fix a nonzero projection $p\in \overline{A}^{\tau }$ and $%
\varepsilon ,\varepsilon _{0}>0$. By (4)$\Rightarrow $(1) in \cite[Lemma 4.2]%
{kerr_turbulence_2010}, it is enough to prove that there exists a projection 
$q\in \overline{A}^{\tau }$ such that 
\begin{equation*}
q\leq p,\ \ \Vert q\overline{\theta }^{\tau }(q)\Vert _{\tau }<\varepsilon
,\ \ \mbox{ and }\ \ \tau (q)\geq \frac{1}{3}\tau (p)>\frac{1}{3}\tau
(a)-\varepsilon .
\end{equation*}%
Let $a\in A_{0}$ satisfy $\left\Vert \pi _{\tau }(a)-p\right\Vert _{\tau
}<\varepsilon $ and $\left\Vert a^{2}-a\right\Vert _{\tau }<\varepsilon $.
By assumption, there exists a contraction $b\in A_{\mathrm{sa}}$ with 
\begin{equation*}
\Vert b^{2}-b\Vert _{\tau }<\varepsilon ,\ \ \Vert ab-b\Vert _{\tau }<\varpi
(\varepsilon )\ \ \Vert b\theta (b)\Vert _{\tau }<\varepsilon ,\ \ 
\mbox{
and }\ \ \tau (b)>\frac{1}{3}\tau (a)-\varepsilon .
\end{equation*}%
By the choice of $\varpi $, there exists a projection $r\in \overline{A}%
^{\tau }$ such that $\left\Vert \pi _{\tau }(r)-b\right\Vert _{\tau }<\varpi
(\varepsilon )$. By choosing $\varepsilon $ small enough, one can ensure
that 
\begin{equation*}
\left\Vert pr-r\right\Vert _{\tau }<\varepsilon _{0},\ \ \left\Vert
rp-p\right\Vert _{\tau }<\varepsilon _{0},\ \ \left\Vert r\overline{\theta }%
^{\tau }(r)\right\Vert _{\tau }<\varepsilon _{0},\ \ \mbox{ and }\ \ \tau
(r)>\frac{1}{3}\tau (a)-\varepsilon _{0}.
\end{equation*}%
By choosing $\varepsilon _{0}$ small enough, one can then find a projection $%
q\in \overline{A}^{\tau }$ satisfying the conditions in item (4) of \cite[%
Lemma 4.2]{kerr_turbulence_2010} mentioned above. This concludes the proof.
\end{proof}

For convenience, we record the following easy lemma. For a relation $%
\mathcal{\mathcal{R}}\subset X\times Y$, its \emph{projection onto $X$} is 
\begin{equation*}
\mathrm{proj}_{X}(\mathcal{R})=\left\{ x\in X\colon \mbox{ there is } y\in Y %
\mbox{ with } ( x,y) \in \mathcal{R}\right\} .
\end{equation*}

\begin{lemma}
\label{Lemma:projection} Let $X$ be a Polish space, let $Y$ be a compact
metrizable space, and let $\mathcal{R}\subset X\times Y$ be a subset. If $%
\mathcal{R}$ is closed, then $\mathrm{proj}_{X}(\mathcal{R})$ is closed. If $%
\mathcal{R}$ is $F_{\sigma }$, then $\mathrm{proj}_{X}(\mathcal{R})$ is $%
F_{\sigma }$.
\end{lemma}

\begin{proof}
It is enough to prove the first assertion, so assume that $\mathcal{R}$ is
closed. Let $( x_{n}) _{n\in \mathbb{N}}$ be a sequence in $\mathrm{proj}%
_{X}(\mathcal{R})$ converging to $x\in X$. Our goal is to prove that $x\in 
\mathrm{proj}_{X}(\mathcal{R})$. For every $n\in \mathbb{N}$, let $y_{n}\in
Y $ satisfy $( x_{n},y_{n}) \in \mathcal{R}$. Since $Y$ is compact, after
passing to a subsequence, we can assume that the sequence $( y_{n})_{n\in%
\mathbb{N}} $ converges to some $y\in Y$. Since $\mathcal{R}$ is closed, we
have $( x,y) \in \mathcal{R}$ and hence $x\in \mathrm{proj}_{X}(\mathcal{R})$%
, as desired.
\end{proof}

Recall the definitions of the sets $\mathrm{SO}_\Lambda(A) $ and $\mathrm{SOWM}%
_\Lambda(A) $ from Notation~\ref{nota:SpacesActions}.

\begin{proposition}
\label{Proposition:free-G} Let $A$ be a unital, separable C*-algebra, let $%
\Lambda$ be a countable group, and let $\tau$ be a trace on $A$. Then the
sets $\mathrm{SO}_\Lambda(A) $ and $\mathrm{SOWM}_\Lambda(A, \tau ) $ are $%
G_{\delta }$ subsets of $\mathrm{Act}_\Lambda(A) $.
\end{proposition}

\begin{proof}
Fix $\gamma \in \Lambda $. Let $\mathcal{R}_{\gamma }$ be the set of pairs $%
\left( \alpha ,\tau \right) \in \mathrm{\mathrm{Act}}_{\Lambda }\left(
A\right) \times \mathrm{T}\left( A\right) $ such that $\tau $ is $\alpha
_{\gamma }$-invariant and $\overline{\alpha }_{\gamma }^{\tau }$ is \emph{not%
} properly outer. By Lemma~\ref{Lemma:free}, $\mathcal{R}_{\gamma }$ is an $%
F_{\sigma }$ subset of $\mathrm{Aut}_{\Lambda}(A)$. By Lemma~\ref%
{Lemma:projection}, its projection $\mathcal{P}_{\gamma }$ onto $\mathrm{Act}%
_{\Lambda }\left( A\right) $ is $F_{\sigma }$ as well. Let $\mathcal{C}%
_{\gamma }$ be the complement of $\mathcal{P}_{\gamma }$ in $\mathrm{Act}%
_{\Lambda }\left( A\right) $, which is $G_{\delta }$. We have that $\mathrm{SO%
}_{\Lambda }\left( A\right) $ is the intersection of $\mathcal{C}_{\gamma }$
for $\gamma \in \Lambda $, and hence $G_{\delta }$. We have already observed
that $\mathrm{WM}_{\Lambda }(A,\tau )$ is $G_{\delta }$. Therefore $\mathrm{%
FWM}_{\Lambda }(A,\tau )=\mathrm{WM}_{\Lambda }(A,\tau )\cap \mathrm{SO}%
_{\Lambda }(A)$ is $G_{\delta }$ as well.
\end{proof}

Adopt the notation of the lemma above. We regard $\mathrm{SO}_{\Lambda }(A)$
as the Polish space of strongly outer actions of $\Lambda $ on $A$.
Consistently, we regard the classification problems for strongly outer
actions of $\Lambda $ on $A$ up to conjugacy, cocycle conjugacy, or weak
cocycle conjugacy, as equivalence relations on $\mathrm{SO}_{\Lambda }(A)$.
Similarly, if $\tau $ is a trace on $A$, we regard $\mathrm{SOWM}_{\Lambda
}(A,\tau )$ as the space of $\tau $-preserving weakly $\tau $-mixing
strongly outer actions of $\Lambda $ on $A$. On the latter space we can also
consider the relations of $\tau $-conjugacy, cocycle $\tau $-conjugacy, and
outer $\tau $-conjugacy.

\subsection{Parametrizing abelian pro-\texorpdfstring{$p$}{p} groups\label%
{Sbs:groups}}

Fix a prime number $p$. In this subsection, we define a compact metrizable
space parametrizing in a canonical way all second countable abelian pro-$p$
groups. The construction is analogous to the one from \cite[Section 2.2]%
{nies_complexity_2016}.

Let $\mathbb{Z}^{\infty }$ be the free abelian group on a countably infinite
set $\{ x_{k}\colon k\in \mathbb{N}\} $ of generators. Let $\mathcal{N}$ be
the (\emph{countable}) collection of finite index subgroups of $\mathbb{Z}%
^{\infty }$ whose index is a multiple of $p$, and which contain all but
finitely many of the generators of $\mathbb{Z}^{\infty }$. We consider $%
\mathbb{Z}^{\infty }$ as a topological group having the elements of $%
\mathcal{N}$ as basis of neighborhoods of the identity. Define $\mathbb{\hat{%
Z}}_{p}^{\infty }$ to be the completion of $\mathbb{Z}^{\infty } $ with
respect to such a topology, which is a second countable abelian pro-$p $
group. In the terminology of \cite[Section 3.3]{ribes_profinite_2010}, $%
\mathbb{\hat{Z}}_{p}^{\infty }$ is the free abelian pro-$p$ group on a
sequence of generators $(x_{k})_{k\in \mathbb{N}}$ converging to $1$.

Suppose that $G$ is a second countable abelian pro-$p$ group. By \cite[%
Proposition 2.4.4 and Proposition 2.6.1]{ribes_profinite_2010} $G$ has a
generating sequence converging to the identity. It therefore follows from 
\cite[Section 3.3.16]{ribes_profinite_2010} that there exists a surjective
continuous group homomorphism $\pi \colon \mathbb{\hat{Z}}_{p}^{\infty
}\rightarrow G$. In other words, $G$ is isomorphic to the quotient of $%
\mathbb{\hat{Z}}_{p}^{\infty }$ by a closed subgroup. Conversely, any
quotient of $\mathbb{\hat{Z}}_{p}^{\infty }$ by a closed subgroup is a
second-countable abelian pro-$p$ group. Thus the closed subgroups of $%
\mathbb{\hat{Z}}_{p}^{\infty }$ naturally parametrize all second-countable
abelian pro-$p$ groups.

We let $\mathcal{K}(\mathbb{\hat{Z}}_{p}^{\infty })$ be the space of closed
subsets of $\mathbb{\hat{Z}}_{p}^{\infty }$ endowed with the Vietoris
topology \cite[Section 4.F]{kechris_classical_1995}, which turns it into a
compact metrizable space. Let also $\mathcal{S}(\mathbb{\hat{Z}}_{p}^{\infty
})\subseteq \mathcal{K}(\mathbb{\hat{Z}}_{p}^{\infty })$ be the (closed)
subset of closed subgroups of $\mathbb{\hat{Z}}_{p}^{\infty }$. Then $%
\mathcal{S}(\mathbb{\hat{Z}}_{p}^{\infty })$ is a compact metrizable space
with the relative topology. We regard isomorphism of second-countable
abelian pro-$p$ groups as an equivalence relation on $\mathcal{S}(\mathbb{%
\hat{Z}}_{p}^{\infty })$.

\begin{proposition}
Let $p$ be a prime number. The relation of topological isomorphism of
second-countable abelian pro-$p$ groups is a complete analytic set.
\end{proposition}

\begin{proof}
As it is observed in \cite[Section 4]{nies_complexity_2016}, Pontryagin's
duality theorem in the special case of profinite \emph{abelian} groups is
witnessed by a Borel map, and for $p$-groups, the parametrization given in 
\cite[Section 4]{nies_complexity_2016} is compatible with the one discussed
above, as we now show.

In \cite[Section 4]{nies_complexity_2016}, abelian profinite groups are
parametrized as follows. Let $F_{\omega }$ be the free group on a countably
infinite set of generators, let $\hat{F}_{\omega }$ be the completion of $%
F_{\omega }$ with respect to the collection of finite index subgroups of $%
F_{\omega }$ which contain all but finitely many of the generators. and let $%
\pi :\hat{F}_{\omega }\rightarrow \mathbb{\hat{Z}}_{p}^{\infty }$ be the
canonical quotient mapping that sends generators to generators. One denotes
by $\mathcal{N}_{\mathrm{ab}}(\hat{F}_{\omega })$ the space of closed normal
subgroup of $\hat{F}_{\omega }$ that contain the commutator subgroup of $%
\hat{F}_{\omega }$. This is a closed subspace of the space of closed subsets
of $\hat{F}_{\omega }$ endowed with the Vietoris topology. Since any second
countable profinite abelian group is a quotient of $\hat{F}_{\omega }$ by an
element of $\mathcal{N}_{\mathrm{ab}}(\hat{F}_{\omega })$, the space $%
\mathcal{N}_{\mathrm{ab}}(\hat{F}_{\omega })$ can be seen as a
parametrization of (presentations of) abelian profinite groups. In this
parametrization, the class of abelian pro-$p$ groups correspond to the
closed subspace $\mathcal{N}_{\mathrm{ab}}(\hat{F}_{\omega })_{p}$ of
elements of $\mathcal{N}_{\mathrm{ab}}(\hat{F}_{\omega })$ that contain $%
\mathrm{\mathrm{Ker}}\left( \pi \right) $. Furthermore, the assignments $%
\mathcal{N}_{\mathrm{ab}}(\hat{F}_{\omega })_{p}\rightarrow \mathcal{S}(%
\mathbb{\hat{Z}}_{p}^{\infty })$, $A\mapsto \pi \left( A\right) $ and $%
\mathcal{S}(\mathbb{\hat{Z}}_{p}^{\infty })\rightarrow \mathcal{N}_{\mathrm{%
ab}}(\hat{F}_{\omega })_{p}$, $A\mapsto \pi ^{-1}\left( A\right) $ are Borel
functions that map presentations for a given abelian pro-$p$ in one
parametrization to presentations for the same abelian pro-$p$ group in the
other parametrization. This shows that the parametrization for abelian pro-$%
p $ groups introduced above is compatible with the parametrization of
arbitrary abelian profinite groups considered in \cite[Section 4]%
{nies_complexity_2016}.

The duals of abelian pro-$p$ groups are precisely the countable abelian $p$%
-groups; see \cite[Theorem 2.9.6 and Lemma 2.9.3]{ribes_profinite_2010}.
Therefore, the relation of isomorphism of countable abelian $p$-groups is
Borel reducible (in fact, Borel isomorphic) to the relation of isomorphism
of second-countable abelian pro-$p$ groups. Since the relation of
isomorphism of countable abelian $p$-groups is a complete analytic set \cite[%
Theorem 6]{friedman_borel_1989}, the result follows.
\end{proof}

\subsection{Reducing groups to actions\label{Sbs:reduction}}

In this last subsection, we obtain the main results of this work. Recall
that for a discrete group $\Lambda $, a separable C*-algebra $A$, and a
trace $\tau $ on $A$, we denote by $\mathrm{SOWM}_{\Lambda }(A,\tau )$ the
Polish space of $\tau $-preserving strongly outer weakly $\tau $-mixing
actions of $\Lambda $ on $A$. Below, we will assume all the C*-algebras to
be \emph{separable}.\emph{\ }In the proof of the following theorem we will
tacitly use the fact---proved in \cite%
{farah_descriptive_2012,farah_turbulence_2014,gardella_conjugacy_2016}%
---that tensor products, direct limits, and crossed products of C*-algebras
and C*-dynamical systems are given by Borel functions with respect to the
parameterizations of C*-algebras and C*-dynamical systems considered in \cite%
{farah_descriptive_2012,farah_turbulence_2014,gardella_conjugacy_2016}.\ It
is also not difficult to see that fixed point algebras of actions of \emph{%
compact }groups on C*-algebras can be computed in a Borel way.

\begin{theorem}
\label{Theorem:conjugacy2}Let $\Lambda $ be a countable group containing an
infinite relative property \emph{(T)} subgroup. Fix a prime number $p$. Let $%
A$ be a separable, locally reflexive, $M_{p^{\infty }}$-absorbing, unital
C*-algebra with an amenable trace, satisfying $A\cong A^{\otimes \mathbb{N}}$%
. Then there exists an extreme, amenable trace $\tau $ on $A$ such that the
relation of isomorphism of second-countable abelian pro-$p$ groups is Borel
reducible to the following equivalence relations on $\mathrm{SOWM}_{\Lambda
}(A,\tau )$:

\begin{enumerate}
\item conjugacy;

\item cocycle conjugacy;

\item weak cocycle conjugacy;

\item $\tau $-conjugacy;

\item cocycle $\tau $-conjugacy;

\item weak cocycle $\tau $-conjugacy
\end{enumerate}
\end{theorem}

\begin{proof}
In view of Theorem \ref{Theorem:conjugacy}, and parts~(4) and~(5) of
Proposition~\ref{prop:AlternativeDescrAlpha}, it is enough to prove that the
function $G\mapsto \alpha ^{G}$ that assigns to a second-countable abelian
pro-$p$ group $G$ the action $\alpha ^{G}:\Lambda \rightarrow \mathrm{Aut}%
\left( A\right) $ from Definition \ref{df:alpha}, is given by a Borel
function with respect to the parametrization of second-countable abelian pro-%
$p$ groups and strongly outer actions of $\Lambda $ on $A$ described in
Subsection \ref{Sbs:actions} and Subsection \ref{Sbs:groups}.

Recall that the action $\alpha ^{G}$ is defined by 
\begin{equation*}
\alpha _{\gamma }^{G}=\xi _{G}^{-1}\circ (\beta _{\Lambda \curvearrowright
\Lambda ,M_{p}^{\infty }}|_{E_{G}}\otimes \sigma )_{\gamma }\circ \xi _{G}
\end{equation*}%
for $\gamma \in \Lambda $, for some choice of isomorphism $\xi
_{G}:E_{G}\otimes A\rightarrow A$. It is therefore enough to show that

\begin{enumerate}
\item the assignment $G\mapsto E_{G}$ is given by a Borel function, and

\item the isomorphism $\xi _{G}:E_{G}\otimes A\rightarrow A$ can be chosen
in a Borel fashion from $E_{G}$.
\end{enumerate}

We address the second assertion first. Several Borel parameterizations of
separable unital C*-algebras are considered in \cite{farah_turbulence_2014}.
Therein, these parameterizations are shown to be equivalent, in the sense
that one can find Borel functions between any two of them, that map a code
for a C*-algebra in one parametrization to a code for the same C*-algebra in
the other parametrization. Furthermore, it is shown in \cite%
{farah_turbulence_2014,farah_descriptive_2012,gardella_conjugacy_2016} that
the standard constructions of C*-algebra theory, including tensor products
and direct limits, are given by Borel functions with respect to these
parametrizations.

For simplicitly, we consider here the parametrization $\Xi $ from \cite%
{farah_turbulence_2014}, which is defined as follows. Let $\mathbb{Q}(i)$ be
the field of Gaussian rationals, and let $\mathcal{U}$ be the collection of
noncommutative *-polynomials with constant term and with coefficients from $%
\mathbb{Q}( i) $ in the variables $(x_{n})_{n\in \mathbb{N}}$.
Let $\Xi $ be the set of functions $f:\mathcal{U}\rightarrow \mathbb{R}$
such that $f$ defines a seminorm on $\mathcal{U}$ with the property that, if 
$C^{\ast }( f) $ is the Hausdorff completion of $\mathcal{U}$
with respect to the metric defined by $f$, then the unital $\mathbb{Q}(
i) $-$\ast $-algebra structure of $\mathcal{U}$ induces a unital
C*-algebra structure on $C^{\ast }( f) $. For $\mathfrak{p}\in 
\mathcal{U}$, we let $\mathfrak{p}_{f}$ be the corresponding element of $%
C^{\ast }( f) $. It is shown in \cite{farah_turbulence_2014} that 
$\Xi $ is a $G_{\delta }$ subset of $\mathbb{R}^{\mathcal{U}}$ endowed with
the product topology. Furthermore, it follows from stability of the
relations defining the matrix units for a unital copy of $M_{p^{n}}$ that
the set \textrm{UHF}$_{p^{\infty }}$ of codes $f\in \Xi $ such that $C^{\ast
}(f)$ is isomorphic to $M_{p^{\infty }}$ is a $G_{\delta }$ subset of $\Xi $%
, and hence a Polish space with the induced topology. Given $f\in \mathrm{UHF%
}_{p^{\infty }}$, a *-isomorphism $\psi :C^{\ast }( f)
\rightarrow M_{p^{\infty }}$ can be regarded as an element of $(M_{p^{\infty
}})^{\mathcal{U}}$. Indeed, given $\psi :C^{\ast }( f)
\rightarrow M_{p^{\infty }}$ one can consider the element $\bar{a}=( a_{%
\mathfrak{p}}) _{\mathfrak{p}\in \mathcal{U}}$ of $(M_{p^{\infty }})^{%
\mathcal{U}}$ defined by setting $a_{\mathfrak{p}}=\psi (\mathfrak{p}_{f})$
for $\mathfrak{p}\in \mathcal{U}$. Thus, it suffices to show that there
exists a Borel assignment \textrm{UHF}$_{p^{\infty }}\rightarrow (
M_{p^{\infty }}) ^{\mathcal{U}}$, $f\mapsto \bar{a}^{(f)}=(a_{%
\mathfrak{p}}^{(f)})_{\mathfrak{p}\in \mathcal{U}}$ such that the assignment 
$\mathfrak{p}_{f}\mapsto a_{\mathfrak{p}}^{(f)}$ extends to a *-isomorphism
from $C^{\ast }( f) $ to $M_{p^{\infty }}$. To this purpose, we
consider the set $\mathcal{A}$ of pairs $( f,(a_{\mathfrak{p}})_{%
\mathfrak{p}\in \mathcal{U}}) \in \mathrm{UHF}_{p^{\infty }}\times
(M_{p^{\infty }})^{\mathcal{U}}$ such that the assignment $\mathfrak{p}%
_{f}\mapsto a_{\mathfrak{p}}$ extends to a *-isomorphism from $C^{\ast
}( f) $ to $M_{p^{\infty }}$. It is easy to see that $\mathcal{A}$
is a $G_{\delta }$ subset of $\mathrm{UHF}_{p^{\infty }}\times (M_{p^{\infty
}})^{\mathcal{U}}$. Furthermore, the automorphism group $\mathrm{Aut}(
M_{p^{\infty }}) $ of $M_{p^{\infty }}$ naturally acts on $%
(M_{p^{\infty }})^{\mathcal{U}}$, in such a way that, for every $f\in 
\mathrm{UHF}_{p^{\infty }}$, the corresponding fiber 
\begin{equation*}
\mathcal{A}_{f}=\left\{ \bar{a}\in (M_{p^{\infty }})^{\mathcal{U}}:( f,%
\bar{a}) \in \mathcal{A}\right\} 
\end{equation*}%
forms a single orbit under the $\mathrm{Aut}( M_{p^{\infty }}) $%
-action. It follows form these observations and \cite[Theorem A]%
{elliott_isomorphism_2013} that $\mathcal{A}$ admits a \emph{Borel
uniformization}, i.e.\ there exists a Borel function \textrm{UHF}$%
_{p^{\infty }}\rightarrow ( M_{p^{\infty }}) ^{\mathcal{U}}$, $%
f\mapsto \bar{a}^{(f)}$ such that $( f,\bar{a}^{(f)}) \in 
\mathcal{A}$ for every $f\in \mathrm{UHF}_{p^{\infty }}$. This allows one to
choose in a Borel fashion, given a C*-algebra $E$ abstractly isomorphic to $%
M_{p^{\infty }}$, a *-isomorphism $\psi _{E}:E\rightarrow M_{p^{\infty }}$.
Since $M_{p^{\infty }}\otimes A$ is isomorphic to $A$, by fixing an
isomorphism $M_{p^{\infty }}\otimes A\cong A$ beforehand, one can choose in
a Borel fashion an isomorphism $\xi _{E}\colon E\otimes A\rightarrow A$.
This justifies the second assertion.

We now justify the first assertion. Recall that $E_{G}$ is the fixed point
algebra of the action $(\delta ^{G})^{\otimes \Lambda }\colon G\rightarrow 
\mathrm{Aut}(D_{G}^{\otimes \Lambda })$. Furthermore, $(\delta
^{G})^{\otimes \Lambda }$ is conjugate to canonical model action $\delta
^{G}:G\rightarrow \mathrm{\mathrm{Aut}}\left( D_{G}\right) $ of $G$
constructed in Theorem \ref{thm:ModelActionProfinite}. Therefore, it is
enough to show that $\delta ^{G}$ can be constructed in a Borel fashion from 
$G$. This is clear when $G$ is finite in view of Remark \ref%
{rmk:modelFiniteG}. In the general case, consider the following. In our
parametrization, a second-countable abelian pro-$p$ group $G$ is given as
the quotient $\mathbb{\hat{Z}}_{p}^{\infty }/N$ of $\mathbb{\hat{Z}}%
_{p}^{\infty }$ by some closed subgroup $N$ of $\mathbb{\hat{Z}}_{p}^{\infty
}$. The finite-index closed subgroups of $G$ correspond to finite-index
closed subgroups $H$ of $\mathbb{\hat{Z}}_{p}^{\infty }$ that contain $N$.
By the Kuratowski--Ryll-Nardzewski selection theorem \cite[Theorem 12.13]%
{kechris_classical_1995}, the collection $\mathcal{V}$ of finite-index
closed subgroups $H$ of $\mathbb{\hat{Z}}_{p}^{\infty }$ that contain $N$
can be chose in a Borel fashion starting from $N$. Since the relation of
inclusion between closed subgroups is closed with respect to the Vietoris
topology, the order on $\mathcal{V}$ given by containment is Borel. This
shows that the canonical inverse system $(G_{i},\pi _{i,j})_{i,j\in \mathcal{%
V}}$ of finite groups having $G$ as inverse limit considered in the proof of
Theorem \ref{thm:ModelActionProfinite} depends on $G$ in a Borel way. The $G$%
-C*-algebra $\left( D_{G},\delta _{G}\right) $ is obtained in the proof of
Theorem \ref{thm:ModelActionProfinite} as the direct limit of the direct
system $\left( \left( D_{G_{i}},\delta _{G_{i}}\right) ,\iota _{ij}\right)
_{i,j\in \mathcal{V}}$, where $\left( D_{G_{i}},\delta _{G_{i}}\right) $ is
the model action of the finite group $G_{i}$. It remains to observe now that
the direct system $\left( \left( D_{G_{i}},\delta _{G_{i}}\right) ,\iota
_{ij}\right) _{i,j\in \mathcal{V}}$ can be computed in a Borel fashion from $%
(G_{i},\pi _{i,j})_{i,j\in \mathcal{V}}$. This concludes the proof.
\end{proof}

\begin{corollary}
\label{Corollary:conjugacy} Under the hypotheses of Theorem \ref%
{Theorem:conjugacy2}, the relations of conjugacy, cocycle conjugacy, and
weak cocycle conjugacy of strongly outer actions of $\Lambda $ on $A$ are
complete analytic sets. Furthermore, there exists an amenable, extreme trace 
$\tau $ on $A$ such that the relations of conjugacy, cocycle conjugacy, weak
cocycle conjugacy, $\tau $-conjugacy, cocycle $\tau $-conjugacy, and weak
cocycle $\tau $-conjugacy of $\tau $-preserving strongly outer weakly $\tau $%
-mixing actions of $\Lambda $ on $A$ are complete analytic sets, and in
particular not Borel.
\end{corollary}

As mentioned after Theorem~\ref{Theorem:uncountably}, it is easy to
construct algebras $A$ satisfied the hypotheses of Corollary~\ref%
{Corollary:conjugacy}. Indeed, if $A_{0}$ is any separable, unital, exact
C*-algebra with an amenable trace, we may take $A=M_{p}^{\infty }\otimes
A_{0}^{\otimes \mathbb{N}}$.


As we did after Theorem~\ref{Theorem:uncountably}, we state the case of
UHF-algebras separately, to highlight the contrast with the main results of 
\cite{kishimoto_rohlin_1995,matui_actions_2011,szabo_strongly_2017}.

\begin{corollary}
\label{cor:UHFanalytic} Let $D$ be a UHF-algebra of infinite type, and let $%
\Lambda$ be a countable group with an infinite subgroup with relative
property \emph{(T)}. Then the relations of conjugacy, cocycle conjugacy, and
weak cocycle conjugacy of strongly outer actions of $\Lambda $ on $D$ are
complete analytic sets, and in particular not Borel. The same applies to the
relations of being conjugate, cocycle conjugate, and weakly cocycle
conjugate in the weak closure with respect to the (necessarily $\Lambda$%
-invariant) unique trace on $D$.
\end{corollary}

In fact, the same conclusions hold for any finite, strongly self-absorbing
C*-algebra containing a nontrivial projection; see \cite[Definition 1.3]%
{toms_strongly_2007} and \cite[Theorem 1.7]{toms_strongly_2007}.

\section{Actions on \texorpdfstring{$R$}{R} with prescribed cohomology}

In this final section, we explain how the methods from Section~\ref%
{Section:conjugacy} can be used to prove Theorem~C in the introduction. In
fact, we prove a somewhat more general statement; see Theorem~\ref%
{thm:PrescrCohom}.

We follow a strategy similar to the one used in Theorem~\ref{thm:ComputCohom}%
, and for that we will need the following replacement of the action
constructed in Theorem~\ref{thm:ModelActionProfinite}. Later on, we will be
interested only in the weak extension of the action constructed below.

Recall that $\mathtt{Lt}\colon G\to \mathrm{Aut}(C(G))$ denotes the action
of left translation. Similarly, we denote by $\mathtt{Rt}$ the action of $G$
on $C(G)$ by right translation.

\begin{proposition}
\label{prop:ModelActCpct} Let $G$ be a second-countable, compact Hausdorff
group. Then there exist a unital C*-algebra $A_{G}$ and an action $\theta
^{G}\colon G\rightarrow \mathrm{Aut}(A_{G})$ with the following properties:

\begin{enumerate}
\item $A_G$ is simple, separable, nuclear, and has a unique trace.

\item There exists an equivariant unital embedding $(C(G), \mathtt{Lt}^G)\to
(A_G, \theta^G)$.

\item $(A_G^{\otimes\mathbb{N}},(\theta^G)^{\otimes\mathbb{N}})$ is
equivariantly isomorphic to $(A_G, \theta^G)$.

\item $\theta^G$ has the Rokhlin property.

\item The fixed point algebra $A_{G}^{\theta ^{G}}$ is a simple, separable,
nuclear, unital C*-algebra with a unique trace.
\end{enumerate}
\end{proposition}

\begin{proof}
Fix a subset $\{x_n\colon n\in\mathbb{N}\}$ of $G$ such that $\{x_n\colon
n\geq m\}$ is dense in $G$ for all $m\in\mathbb{N}$. For $n\in\mathbb{N}$,
set $A_n=M_{2^n}\otimes C(G)$ with the $G$-action $\alpha^{(n)}=\mathrm{id}%
_{M_{2^n}}\otimes\mathtt{Lt}$. Let $\rho^{(n)}$ be the $G$-action on $A_n$
given by $\rho^{(n)}=\mathrm{id}_{M_{2^n}}\otimes\mathtt{Rt}$, and define a
unital equivariant homomorphism $\varphi_n\colon A_n\to A_{n+1}$ by $%
\varphi_n(a)= \left( 
\begin{array}{cc}
a & 0 \\ 
0 & \rho_{x_n}(a) \\ 
& 
\end{array}
\right)$ for all $a\in A_n$. Define $\widetilde{A}_G$ and $\widetilde{\theta}%
^G$ to be the direct limits of $(A_n)_{n\in\mathbb{N}}$ and $%
(\alpha^{(n)})_{n\in\mathbb{N}}$ with respect to the connecting maps $%
(\varphi_n)_{n\in\mathbb{N}}$.

Observe that $\widetilde{A}_G$ is separable, unital and nuclear. We claim
that it has a unique trace and that it is simple. Both fact are proved using
similar arguments, so we only show uniqueness of the trace. Denote by $\tau$
the trace on $C(G)$ given by integration against the normalized Haar measure 
$\mu$ on $G$, and by $\mathrm{tr}_n$ the normalized trace on $M_{2^n}$. Then 
$\tau_n=\mathrm{tr}_n\otimes\tau$ is a normalized trace on $A_n$ and $%
\tau_{n+1}\circ\varphi_n=\tau_n$ for all $n\in\mathbb{N}$. It follows that
there is a direct limit trace on $\widetilde{A}_G$. Now let $\sigma$ be
another trace on $\widetilde{A}_G$. Then there exist $n_0\in\mathbb{N}$ and
traces $\sigma_n\in T(A_n)$, for $n\geq n_0$, satisfying $%
\sigma_{n+1}\circ\varphi_n=\sigma_n$ for all $n\geq n_0$ and $%
\tau(a)=\tau_n(a)$ for all $a\in A_n$, for $n\geq n_0$. For $n\geq n_0$, let 
$\nu_n$ be a probability measure on $C(G)$ such that, with $\widehat{\nu}_n$
denoting its associated functional on $C(G)$, we have $\sigma_n=\mathrm{tr}%
_n\otimes\widehat{\nu}_n$. The identity $\sigma_{n+1}\circ\varphi_n=\sigma_n$
amounts to $\nu_n(E)=\frac{1}{2}\left(\nu_{n+1}(x_nE)+\nu_{n+1}(E)\right)$
for every measurable subset $E\subseteq G$. Using the identity $%
\sigma_{n+k}\circ\varphi_{n+k-1}\circ\cdots\circ\varphi_n=\sigma_n$, valid
for all $k\geq 1$, an using that $\{x_k\colon k\geq n\}$ is dense in $G$,
one concludes that $\nu_n$ is translation invariant, and hence we must have $%
\nu_n=\mu$ for all $n\in\mathbb{N}$. In particular, it follows that $%
\sigma=\tau$, as desired.

The proof of simplicity is analogous, using open subsets of $G$ which are
translation invariant. We omit the details.

Now set $A_G=\otimes_{k\in\mathbb{N}} \widetilde{A}_G$ and $%
\alpha^G=\otimes_{k\in\mathbb{N}} \widetilde{\alpha}^G$. Then $A_G$ is
simple, separable, unital, nuclear, and has a unique trace, which proves
(1). Observe that there are equivariant unital embeddings 
\begin{equation*}
C(G)\hookrightarrow A_1 \hookrightarrow \widetilde{A}_G\hookrightarrow A_G,
\end{equation*}
so part~(2) is satisfied. Also, (3) holds by construction, while (4) follows
from (2) and (3). Finally, the Rokhlin property for $\theta^G$ ensures that
the properties for $A_G$ listed in~(1) are inherited by $A_G^{\theta^G}$, by
the theorem in the introduction of~\cite{gardella_crossed_2014}. This gives
(5), and finishes the proof.
\end{proof}

Observe that $A_{G}$ is never a UHF-algebra (unless $G$ is the trivial
group). In particular, even when $G$ is a profinite group, the C*-dynamical
system $(A_{G},\theta ^{G})$ constructed in Proposition~\ref%
{prop:ModelActCpct} is not the same as that constructed in Theorem~\ref%
{thm:ModelActionProfinite}.

\begin{remark}
Unlike in Theorem~\ref{thm:ModelActionProfinite}, the action constructed in
the proposition above does not enjoy any reasonable uniqueness-type property
among Rokhlin actions of $G$.
\end{remark}

We now come to the main result of this section, which in particular implies
Theorem~C in the introduction.

\begin{theorem}
\label{thm:PrescrCohom} Let $\Lambda$ be a countable group containing an
infinite subgroup $\Delta$ with relative property (T), and let $\Gamma$ be
any countable abelian group. Then there exist an outer action $%
\alpha^{\Gamma}\colon \Lambda\to \mathrm{Aut} (R)$ and bijections $%
\eta\colon H^1_{\Delta,w}(\alpha^{\Gamma})\to \Gamma$ and $\eta^{(2)}\colon
H^1_{\Delta,w}(\alpha^{\Gamma}\otimes\alpha^{\Gamma})\to \Gamma$ making the
following diagram commute: 
\begin{align*}
\xymatrix{ H^1_{\Delta,w}(\alpha^{\Gamma})\times
H^1_{\Delta,w}(\alpha^{\Gamma})
\ar[d]_-{m^{\alpha^{\Gamma}}}\ar[rr]^-{\eta\times \eta} && \Gamma\times
\Gamma \ar[d]^-{m^{\Gamma}}\\
H^1_{\Delta,w}(\alpha^{\Gamma}\otimes\alpha^{\Gamma})\ar[rr]_-{{\eta^{(2)}}}
&& \Gamma }
\end{align*}
\end{theorem}

\begin{proof}
Let $G$ denote the Pontryagin dual of $\Gamma $, which is a
second-countable, compact Hausdorff group. Let $\theta ^{G}\colon
G\rightarrow \mathrm{Aut}(A_{G})$ denote the action constructed in
Proposition~\ref{prop:ModelActCpct}, and denote by $\overline{\theta }%
^{G}\colon G\rightarrow \mathrm{Aut}(R)$ its weak extension in the GNS
representation associated to the unique (and hence $\theta ^{G}$-invariant)
trace of $A_{G}$. (The fact that the weak closure of $A_{G}$ is $R$ follows
from Lemma~\ref{lemma:WeakClosureR} and part~(1) of Proposition~\ref%
{prop:ModelActCpct}.) Abbreviate $R^{\otimes \Lambda }$ to $N$, and
abbreviate $(\overline{\theta }^{G})^{\otimes \Lambda }$ to $\rho \colon
G\rightarrow \mathrm{Aut}(N)$. Denote by $\beta \colon \Lambda \rightarrow 
\mathrm{Aut}(N)$ the Bernoulli shift $\beta _{\Lambda \curvearrowright
\Lambda ,R}$ of $\Lambda $ on $R^{\otimes \Lambda }=N$. Then $\beta $
commutes with $\rho $, and hence induces an action $\alpha ^{\Gamma }$ of $%
\Lambda $ on the fixed point algebra $N^{\rho }$ of $\rho $. Since $(\theta
^{G})^{\otimes \Lambda }$ is conjugate to $\theta ^{G}$ by part~(3) of
Proposition~\ref{prop:ModelActCpct}, it follows that $N$, which is the weak
closure of the fixed point algebra of $(\theta ^{G})^{\otimes \Lambda }$, is
isomorphic to the weak closure of $A_{G}^{\theta ^{G}}$. Hence $N^{\rho }$
is isomorphic to $R$ by part~(5) of Proposition~\ref{prop:ModelActCpct} and
Lemma~\ref{lemma:WeakClosureR}. Under this identification, we regard $\alpha
^{\Gamma }$ as an action of $\Lambda $ on $R$. Finally, the same proof as
Lemma \ref{Lemma:key} gives the desired conclusion concerning the $\Delta $%
-relative weak cohomology group of $\alpha ^{\Gamma }$.
\end{proof}

We close this work by pointing out that the argument used in Theorem~\ref%
{Theorem:uncountably} can be used in this context to give an alternative
proof of Theorem~B in \cite{brothier_families_2015} for the case of groups
containing a subgroup with the relative property (T). Namely, it follows
from \autoref{thm:PrescrCohom} that for $\Lambda$ as in its assumptions,
there exist uncountably many weakly non-cocycle conjugate outer actions of $%
\Lambda$ on $R$.

\providecommand{\MR}[1]{}
\providecommand{\bysame}{\leavevmode\hbox to3em{\hrulefill}\thinspace}
\providecommand{\MR}{\relax\ifhmode\unskip\space\fi MR }
\providecommand{\MRhref}[2]{%
  \href{http://www.ams.org/mathscinet-getitem?mr=#1}{#2}
}
\providecommand{\href}[2]{#2}


\begin{thebibliography}{10}

\bibitem{becker_descriptive_1996}
Howard Becker and Alexander~S. Kechris, \emph{The descriptive set theory of
  {P}olish group actions}, London {Mathematical} {Society} {Lecture} {Note}
  {Series}, vol. 232, Cambridge University Press, Cambridge, 1996.

\bibitem{blackadar_operator_2006}
Bruce Blackadar, \emph{Operator {algebras}}, Encyclopaedia of {Mathematical}
  {Sciences}, vol. 122, Springer-Verlag, Berlin, 2006.

\bibitem{bratteli_rohlin_1995}
O.~Bratteli, D.~E. Evans, and A.~Kishimoto, \emph{The {Rohlin} {Property} {For}
  {Quasi}-{Free} {Automorphisms} of the {Fermion} {Algebra}}, Proceedings of
  the London Mathematical Society \textbf{s3-71} (1995), no.~3, 675--694 (en).

\bibitem{brothier_families_2015}
Arnaud Brothier and Stefaan Vaes, \emph{Families of hyperfinite subfactors with
  the same standard invariant and prescribed fundamental group}, Journal of
  Noncommutative Geometry \textbf{9} (2015), no.~3, 775--796.

\bibitem{brown_invariant_2006}
Nathanial~P. Brown, \emph{Invariant means and finite representation theory of
  {C}*-algebras}, Mem. Amer. Math. Soc. \textbf{184} (2006), no.~865, viii+105.
  \MR{2263412}

\bibitem{brown_c*-algebras_2008}
Nathanial~P. Brown and Narutaka Ozawa, \emph{C*-algebras and finite-dimensional
  approximations}, Graduate {Studies} in {Mathematics}, vol.~88, American
  Mathematical Society, Providence, RI, 2008.

\bibitem{burger_kazhdan_1991}
Marc Burger, \emph{Kazhdan constants for {$\mathrm{SL}(3,\mathbb{Z})$}},
  Journal f{\"{u}}r die reine und angewandte Mathematik \textbf{413} (1991),
  36--67. \MR{1089795}

\bibitem{connes_outer_1975}
Alain Connes, \emph{Outer conjugacy classes of automorphisms of factors},
  Annales Scientifiques de l'{\'{E}}cole Normale Sup{\'{e}}rieure.
  Quatri{\`{e}}me S{\'{e}}rie \textbf{8} (1975), no.~3, 383--419. \MR{0394228}

\bibitem{connes_periodic_1977}
\bysame, \emph{Periodic automorphisms of the hyperfinite factor of type
  {II$_1$}}, Acta Universitatis Szegediensis. Acta Scientiarum Mathematicarum
  \textbf{39} (1977), no.~1-2, 39--66. \MR{0448101}

\bibitem{elliott_isomorphism_2013}
George~A. Elliott, Ilijas Farah, Vern~I. Paulsen, Christian Rosendal, Andrew~S.
  Toms, and Asger Törnquist, \emph{The isomorphism relation for separable
  {C}*-algebras}, Mathematical Research Letters \textbf{20} (2013), no.~6,
  1071--1080.

\bibitem{epstein_borel_2011}
Inessa Epstein and Asger T{\"{o}}rnquist, \emph{The {Borel} complexity of von
  {Neumann} equivalence}, arXiv:1109.2351 (2011).

\bibitem{farah_descriptive_2012}
Ilijas Farah, Andrew~S. Toms, and Asger T{\"o}rnquist, \emph{The descriptive
  set theory of {C}*-algebra invariants}, International Mathematics Research
  Notices (2012), 5196--5226.

\bibitem{farah_turbulence_2014}
\bysame, \emph{Turbulence, orbit equivalence, and the classification of nuclear
  {C}*-algebras}, Journal f{\"u}r die reine und angewandte Mathematik
  \textbf{688} (2014), 101--146.

\bibitem{friedman_borel_1989}
Harvey Friedman and Lee Stanley, \emph{A {Borel} reducibility theory for
  classes of countable structures}, Journal of Symbolic Logic \textbf{54}
  (1989), no.~3, 894--914.

\bibitem{gao_invariant_2009}
Su~Gao, \emph{Invariant {descriptive} {set} {theory}}, Pure and {Applied}
  {Mathematics}, vol. 293, CRC Press, Boca Raton, FL, 2009.

\bibitem{gardella_crossed_2014}
Eusebio Gardella, \emph{Crossed products by compact group actions with the
  {Rokhlin} property}, Journal of Noncommutative Geometry, in press.

\bibitem{gardella_rokhlin_2014}
\bysame, \emph{Rokhlin dimension for compact group actions}, Indiana Journal of
  Mathematics, \textbf{66} (2017), 659--703.

\bibitem{gardella_compact_2015}
\bysame, \emph{Compact group actions on {C}*-algebras: classification,
  non-classifiability, and crossed products; and rigidity results for
  ${L}^p$-operator algebras}, Ph.D. thesis, University of Oregon, 2015.

\bibitem{gardella_conjugacy_2016}
Eusebio Gardella and Martino Lupini, \emph{Conjugacy and cocycle conjugacy of
  automorphisms of {$\mathcal{O}_2$} are not {Borel}}, M{\"{u}}nster Journal of
  Mathematics \textbf{9} (2016), no.~1, 93--118.

\bibitem{hirshberg_rokhlin_2007}
Ilan Hirshberg and Wilhelm Winter, \emph{Rokhlin actions and self-absorbing
  {C}*-algebras}, Pacific Journal of Mathematics \textbf{233} (2007), no.~1,
  125--143.

\bibitem{izumi_finite_2004-1}
Masaki Izumi, \emph{Finite group actions on {C}*-algebras with the {Rohlin}
  property---{II}}, Advances in Mathematics \textbf{184} (2004), no.~1,
  119--160.

\bibitem{izumi_finite_2004}
\bysame, \emph{Finite group actions on {C}*-algebras with the {Rohlin}
  property, {I}}, Duke Mathematical Journal \textbf{122} (2004), no.~2,
  233--280.

\bibitem{jiang_simple_1999}
Xinhui Jiang and Hongbing Su, \emph{On a {simple} {unital} {projectionless}
  {C}*-{algebra}}, American Journal of Mathematics \textbf{121} (1999), no.~2,
  359--413.

\bibitem{jolissaint_property_2005}
Paul Jolissaint, \emph{On property ({T}) for pairs of topological groups},
  L'Enseignement Math{\'{e}}matique \textbf{51} (2005), no.~1-2, 31--45.
  \MR{2154620}

\bibitem{jones_actions_1980}
Vaughan F.~R. Jones, \emph{Actions of finite groups on the hyperfinite type
  {II$_1$} factor}, Memoirs of the American Mathematical Society \textbf{28}
  (1980), no.~237, v+70. \MR{587749}

\bibitem{jones_converse_1983}
\bysame, \emph{A converse to {Ocneanu}'s theorem}, Journal of Operator Theory
  \textbf{10} (1983), no.~1, 61--63. \MR{715556}

\bibitem{kallman_generalization_1969}
Robert~R. Kallman, \emph{A generalization of free action}, Duke Mathematical
  Journal \textbf{36} (1969), 781--789. \MR{0256181}

\bibitem{kechris_classical_1995}
Alexander Kechris, \emph{Classical {descriptive} {set} {theory}}, Graduate
  {Texts} in {Mathematics}, vol. 156, Springer-Verlag, New York, 1995.

\bibitem{kechris_amenable_2008}
Alexander Kechris and Todor Tsankov, \emph{Amenable actions and almost
  invariant sets}, Proceedings of the American Mathematical Society
  \textbf{136} (2008), no.~2, 687--697.

\bibitem{kerr_turbulence_2010}
David Kerr, Hanfeng Li, and Mika{\"{e}}l Pichot, \emph{Turbulence,
  representations, and trace-preserving actions}, Proceedings of the London
  Mathematical Society \textbf{100} (2010), no.~2, 459--484.

\bibitem{kishimoto_rohlin_1995}
Akitaka Kishimoto, \emph{The {Rohlin} property for automorphisms of {UHF}
  algebras}, Journal f{\"{u}}r die reine und angewandte Mathematik
  \textbf{1995} (1995), no.~465, 183--196.

\bibitem{margulis_finitely-additive_1982}
Gregory~A. Margulis, \emph{Finitely-additive invariant measures on {Euclidean}
  spaces}, Ergodic Theory and Dynamical Systems \textbf{2} (1982), no.~3-4,
  383--396 (1983).

\bibitem{matui_actions_2011}
Hiroki Matui, \emph{{$\mathbb{Z}^N$}-actions on {UHF} algebras of infinite
  type}, Journal f{\"{u}}r die reine und angewandte Mathematik \textbf{657}
  (2011), 225--244.

\bibitem{matui_stability_2012}
Hiroki Matui and Yasuhiko Sato, \emph{{$\mathcal{Z}$}-stability of crossed
  products by strongly outer actions}, Communications in Mathematical Physics
  \textbf{314} (2012), no.~1, 193--228.

\bibitem{matui_stability_2014}
\bysame, \emph{{$\mathcal{Z}$}-stability of crossed products by strongly outer
  actions {II}}, American Journal of Mathematics \textbf{136} (2014), no.~6,
  1441--1496.

\bibitem{nies_complexity_2016}
Andr{\'{e}} Nies, \emph{The complexity of isomorphism between countably based
  profinite groups}, arXiv:1604.00609 (2016).

\bibitem{ocneanu_actions_1985}
Adrian Ocneanu, \emph{Actions of discrete amenable groups on von {N}eumann
  algebras}, Lecture Notes in Mathematics, vol. 1138, Springer-Verlag, Berlin,
  1985. \MR{807949}

\bibitem{popa_computations_2006}
Sorin Popa, \emph{Some computations of 1-cohomology groups and construction of
  non-orbit-equivalent actions}, Journal of the Institute of Mathematics of
  Jussieu \textbf{5} (2006), no.~2, 309--332. \MR{2225044}

\bibitem{popa_some_2006}
\bysame, \emph{Some rigidity results for non-commutative {Bernoulli} shifts},
  Journal of Functional Analysis \textbf{230} (2006), no.~2, 273--328.

\bibitem{ribes_profinite_2010}
Luis Ribes and Pavel Zalesskii, \emph{Profinite groups}, second ed., {Series}
  of {Modern} {Surveys} in {Mathematics}, vol.~40, Springer-Verlag, Berlin,
  2010.

\bibitem{sato_rohlin_2010}
Yasuhiko Sato, \emph{The {Rohlin} property for automorphisms of the
  {Jiang}-{Su} algebra}, Journal of Functional Analysis \textbf{259} (2010),
  no.~2, 453--476.

\bibitem{szabo_strongly_2017}
G{\'{a}}bor Szab{\'{o}}, \emph{Strongly self-absorbing {C}*-dynamical systems,
  {III}}, Advances in Mathematics \textbf{316} (2017), 356--380.

\bibitem{takesaki_theory_2002}
M.~Takesaki, \emph{Theory of operator algebras. {I}}, Encyclopaedia of
  {Mathematical} {Sciences}, vol. 124, Springer-Verlag, Berlin, 2002.

\bibitem{takesaki_theory_2003}
\bysame, \emph{Theory of operator algebras. {III}}, Encyclopaedia of
  {Mathematical} {Sciences}, vol. 127, Springer-Verlag, Berlin, 2003.

\bibitem{toms_strongly_2007}
Andrew~S. Toms and Wilhelm Winter, \emph{Strongly self-absorbing
  {C}*-algebras}, Transactions of the American Mathematical Society
  \textbf{359} (2007), no.~8, 3999--4029.

\bibitem{tornquist_localized_2011}
Asger T{\"{o}}rnquist, \emph{Localized cohomology and some applications of
  {Popa}'s cocycle superrigidity theorem}, Israel Journal of Mathematics
  \textbf{181} (2011), 327--346. \MR{2773046}

\bibitem{vaes_rigidity_2007}
Stefaan Vaes, \emph{Rigidity results for {Bernoulli} actions and their von
  {Neumann} algebras (after {Sorin} {Popa})}, Ast{\'{e}}risque (2007), no.~311,
  237--294, S{\'{e}}minaire Bourbaki. Vol. 2005/2006. \MR{2359046}

\bibitem{winter_localizing_2014}
Wilhelm Winter, \emph{Localizing the {E}lliott conjecture at strongly
  self-absorbing {C}*-algebras}, Journal f{\"{u}}r die Reine und Angewandte
  Mathematik \textbf{692} (2014), 193--231.

\end{thebibliography}
\end{document}